\documentclass[11pt]{article}

\usepackage[utf8]{inputenc}
\usepackage{amsmath,amsthm,amssymb,enumerate,framed}
\usepackage{color,colortab,epsfig}
\usepackage{graphicx}
\usepackage{psfrag}
\usepackage{subfigure}

\newcommand{\Q}{\mathbb Q}
\newcommand{\N}{\mathbb{N}}
\newcommand{\Z}{\mathbb{Z}}
\newcommand{\R}{\mathbb{R}}
\newcommand{\conv}{\operatorname{conv}}
\newcommand{\vect}[1]{\left(\begin{array}{@{}r@{}}#1\end{array}\right)}
\newcommand{\setcond}[2]{\left\{ #1 \,:\, #2 \right\}}

\newcommand{\st}{:\;}

\newcommand{\intr}{\operatorname{int}}
\newcommand{\relintr}{\operatorname{relint}}
\newcommand{\relint}{\relintr}
\newcommand{\rec}{\operatorname{rec}}
\newcommand{\lin}{\operatorname{lin}}

\newcommand{\cl}{\operatorname{cl}}
\newcommand{\proj}{\operatorname{proj}}

\newcommand{\floor}[1]{\left\lfloor#1\right\rfloor}

\newcommand{\vol}{\operatorname{vol}}

\renewcommand{\epsilon}{\varepsilon}

\newtheoremstyle{mythmstyle}
	{\topsep}
	{\topsep}
	{\itshape}
	{}
	{\scshape}
	{.}
	{3pt}
	{}
\theoremstyle{mythmstyle}

\newtheorem{nn}{}[section]
\newtheorem{lemma}[nn]{Lemma}
\newtheorem{theorem}[nn]{Theorem}
\newtheorem{cor}[nn]{Corollary}
\newtheorem{prop}[nn]{Proposition}
\newtheorem{definition}[nn]{Definition}

\newtheorem{REMARK}[nn]{Remark}
\newenvironment{remark}{\begin{REMARK}}{\end{REMARK}}

\newenvironment{cpf}{\begin{trivlist} \item[] {\em Proof of Claim.}}{\hspace*{\stretch{1}} $\diamond$ \end{trivlist}}

\newtheoremstyle{itsemicolon}{}{}{\mdseries\rmfamily}{}{\itshape}{:}{ }{}
\newtheoremstyle{itdot}{}{}{\mdseries\rmfamily}{}{\itshape}{:}{ }{}
\theoremstyle{itdot}
\newtheorem*{msc*}{2010 Mathematics Subject Classification}
\newtheorem*{keywords*}{Keywords}

\newcommand{\aff}{\operatorname{aff}}

\newcommand{\rx}{{r}}
\newcommand{\x}{{x}}
\newcommand{\y}{{y}}

\newcommand{\g}{{g}}

\renewcommand{\u}{{u}}
\renewcommand{\a}{{a}}

\newcommand{\0}{{0}}

\newcommand{\w}{{w}}

\newcommand{\gp}{{ {\bar g}}}
\newcommand{\gt}{{{\tilde g}}}
\newcommand{\gs}{{  g}}

\newenvironment{example}[1][]{\par\medskip\noindent%
   \textbf{Example~\thetheorem.}~#1 \rmfamily\ignorespaces}{$\blacksquare$\medskip}

\numberwithin{equation}{section}

\title{A geometric approach to cut-generating functions}
\author{Amitabh Basu\footnote{Department of Applied Mathematics and Statistics, The Johns Hopkins University, MD, USA}
\and Michele Conforti\footnote{Dipartimento di Matematica, Universit\`a degli Studi di Padova, Italy. Supported by the grant ``Progetto di Ateneo 2013'' of the University of Padova.}
\and Marco Di Summa\footnotemark[2]}

\begin{document}

\maketitle

\begin{abstract}
The cutting-plane approach to integer programming was initiated more that 40 years ago: Gomory introduced the corner polyhedron as a relaxation of a mixed integer set in tableau form and Balas introduced intersection cuts for the corner polyhedron. This line of research was left dormant for several decades until relatively recently, when a paper of Andersen, Louveaux, Weismantel and Wolsey generated renewed interest in the corner polyhedron and intersection cuts. Recent developments rely on tools drawn from convex analysis, geometry and number theory, and constitute an elegant bridge between these areas and integer programming. We survey these results and highlight recent breakthroughs in this area.
\end{abstract}

\section{Introduction}\label{s:intro}

The cutting-plane approach to integer programming (IP) was initiated in the early 1970s with the works of  Gomory~\cite{MR0102437,gomory1960algorithm,Gomory63,MR0182454,gom,gomory2007atoms}, Gomory and Johnson \cite{infinite,infinite2} and Johnson \cite{johnson}  on the corner polyhedron, and of Balas \cite{bal}
 on intersection cuts generated from convex sets.
Their approach aimed at the development of a theory of valid inequalities (to be used as cuts) for integer programs, pure or mixed, \emph{independently} of the structure and the data  of the problem on hand. Gomory introduced a universal model which provided a relaxation for any integer program defined by constraints in tableau form and studied \emph{cut-generating functions}. These functions when applied to a specific IP problem, provide a valid inequality that is not satisfied by the basic solution associated with the tableau.
  \smallskip

While the point of view of Gomory was algebraic, the approach of Balas was essentially based on the geometry of the sets to be studied.
As an example, split cuts are the simplest and most effective family of intersection cuts. They are equivalent to Gomory's mixed integer (GMI) cuts, which are generated by applying a cut-generating function to a single row of a tableau, see e.g.\@ Theorem 5.5 in \cite{conforti2014integer}.
\smallskip

Possibly inspired by the deep and elegant results in combinatorial optimization and polyhedral combinatorics, research in IP then shifted its focus on the study of strong (facet-defining) valid inequalities for \emph{structured} integer programs of the combinatorial type, mostly 0,1 programs.
\smallskip

Renewed interest in Gomory's approach was recently sparked by a paper of  Andersen, Louveaux, Weismantel and Wolsey \cite{alww}.
They define a relaxation of Gomory's model whose tableaux has two rows.
This is a 2-dimensional model and  can be represented in the plane. They  show that,
besides nonnegativity constraints, the facet-defining inequalities
are naturally associated with splits (a region between two parallel lines), triangles and quadrilaterals whose interior does not contain an integer point.  This allows one to derive valid inequalities by exploiting the combined effect of two rows, instead of a single row.

The extension of this model to any dimension (i.e., number of integer variables) was pioneered by Borozan and Cornu\'ejols~\cite{BorCor} and by Basu, Conforti, Cornu\'ejols and Zambelli~\cite{bccz}. The main finding is that like in the 2-dimensional case, facet-defining inequalities
are naturally associated with full-dimensional convex sets whose interior does not contain an integer point. Furthermore these sets are polyhedra. In this survey we highlight the importance of this fact as it provides a simple formula to compute the associated cut-generating function.

Indeed, Lov\'asz \cite{lovasz} stated that maximal convex sets whose interior does not contain an integer point are polyhedra, but the first proof appears in \cite{bccz}, and an alternate proof can be found in~\cite{MR3027668}. These proofs use the simultaneous approximation theorem of Dirichlet and Minkowski's convex body theorem.  So important results from number theory and convex geometry are fundamental in proving polyhedrality, which is essential to get a computable  formula for the cut-generating function.  We highlight the use of these and other ``classical'' theorems in this survey.
\medskip

This survey first introduces in Section~\ref{sec:separation} a general mixed-integer set which provides a framework to study cut-generating functions. We then highlight in Section~\ref{sec:significance} some areas of applicability of this general mixed-integer set.

The next three sections are essentially devoted to special cases of this general mixed-integer set. Section \ref{sec:continuous-infinite-relax} deals with the case introduced by Andersen et al.\@ \cite{alww}. We first explain how the gauge function, from convex analysis, links convex sets whose interiors do not contain integer points and valid inequalities. We then give a novel and concise proof of a Lov\'asz's characterization of maximal convex sets whose interior does not contain an integer point. It is based on a proposition that characterizes the sets arising as a projection of the integers on a subspace. We then survey extensions of these results to a more general setting that includes complementarity and nonlinear constraints.

Section \ref{sec:pure} considers the special case of the mixed-integer set with only integer variables. This case, first introduced by Gomory and Johnson in~\cite{infinite,infinite2}, was the starting point of the theory of cut-generating functions and is also known as the {\em infinite group relaxation} in the literature. We show recent progress in extending classical results beyond the single-row problem to arbitrary number of rows; in particular, we emphasize the use of the {\em Knaster-Kuratowski-Mazurciewicz lemma}, a classical result from fixed-point theory, in this development.

In Section \ref{sec:mixed}, we discuss the general set with both continuous and integer variables. We focus on the aspect of {\em lifting}, where the quest for computable formulas for cutting planes leads to unexpected connections with the theory of tilings and coverings of Euclidean space. Such questions have classically been considered within the field of geometry of numbers and the recent connection with integer programming leads to a rich theory. We also highlight a recent discovery from~\cite{basu-paat-lifting} that topological arguments, such as the classical {\em Invariance of Domain theorem} from algebraic topology, can lead to important results in this area.

To summarize, in this survey we concentrate on results from the cut-generating function approach to cutting planes. An introduction to this topic can be found in Chapter 6 of~\cite{conforti2014integer}; see also~\cite{corner_survey}. There is a parallel body of work which uses finite cyclic groups to study Gomory's corner polyhedron and cutting planes. Richard and Dey's survey~\cite{Richard-Dey-2010:50-year-survey} on the group approach covers this aspect, as well as its links with cut-generating functions. Further, there has been lot of work in studying closures of families of cutting planes, and convergence issues in cutting plane algorithms. We do not discuss these topics in this survey; the reader is referred to the survey by Del Pia and Weismantel~\cite{delpia-4OR}. A recent survey by Basu, K\"oppe and Hildebrand~\cite{bhk-survey} delves deeper into aspects of cut-generating functions for the pure integer case that are discussed in Section \ref{sec:pure} of this survey.

We believe that the recent results that we survey here, such as the computable formula for cut-generating functions derived in Section \ref{sec:continuous-infinite-relax}, together with the theory of lifting discussed in Section \ref{sec:mixed}, are the first steps towards making cut-generating functions in general dimensions a viable computational tool. Prior to this, only one-dimensional cut-generating functions were explicitly provided with which one could perform computations.

In fact, even the one-dimensional theory developed by Gomory did not find its way into IP solvers for decades and was believed to be of little or no computational use. This point of view changed dramatically in the mid 1990s following the work of Balas, Ceria, Cornu\'ejols, Natraj~\cite{balas96gomory}; see~\cite{cornuejols2007revival} for a recent account. Today most cutting planes currently implemented in software are based on this one-dimensional theory, such as GMI cuts from tableau rows, mixed integer rounding inequalities and lift-and-project cuts \cite{conforti2014integer}. We hope that the corresponding progress for higher-dimensional cut-generating functions can provide another boost to the efficiency of mixed-integer optimization solvers. On the flip side, this poses greater challenges in choosing the ``right'' cutting planes, since this theory significantly increases the pool of available cuts. Computational experiments have been conducted by Dey, Lodi, Tramontani, Wolsey~\cite{dey2010experiments}, Basu, Bonami, Cornu\'ejols, Margot~\cite{basu2011experiments}, and Louveaux and Poirrier~\cite{lp}, based mostly on the special case discussed in Section 4. However, some of the developments surveyed here have not been computationally tested, and the effectiveness of these new findings remains open.


\section{Separation and valid functions}\label{sec:separation}

For fixed $n\in\N$, let $S$ be a closed subset of $\R^n$ that does not contain the origin 0. In this survey, we consider subsets of the following form:
\begin{equation}\label{def mixed-int set}
	X_{S}(R,P) := \setcond{(s,y) \in \R_+^k \times \Z_+^\ell }{ Rs + Py\in S },
\end{equation}
where $k, \ell \in \Z_+$, $R \in \R^{n \times k}$ and $P \in \R^{n \times \ell}$ are matrices. We allow $k=0$ or $\ell=0$, but not both. These sets were first introduced by Johnson in \cite{johnson} and \cite{johnson1981characterization}, based on earlier work by Gomory and Johnson in~\cite{infinite, infinite2}.
We address the following
 \smallskip

\noindent\emph {{\sc Separation problem}: Find a closed half-space  that contains $X_{S}(R,P)$ but not the origin.}\smallskip

The fact that $S$ is closed and $0\notin S$ implies $0$ is not in the closed convex hull of $X_{S}(R,P)$~\cite[Lemma 2.1]{conforti2013cut}. Hence such a half-space always exists.

This problem arises typically when one wants to design a cutting-plane method to optimize a (linear) function over $X_{S}(R,P)$ and has on hand a solution (the origin 0) to a relaxation of the problem (see Section~\ref{sec:significance}).
\smallskip

We develop a theory that for fixed $S$ addresses the separation problem \emph{independently of $R$ and $P$} by introducing the concept of valid pair.

A \emph{valid pair} $(\psi, \pi)$ for $S$ is a pair of functions $\psi, \pi\colon\R^n \to \R$ such that \emph{for every choice of $k$, $\ell$, $R$ and $P$},
\begin{equation}
	\label{psi pi ineq}
	\sum\psi(r)s_r + \sum\pi(p)y_p \ge 1
\end{equation}
is an inequality separating 0 from $X_S(R,P)$ (this is the reason for the choice of 1 for the right hand side).
We use the convention that the first sum is taken over the columns $r$ of $R$, where $s_r$ denotes the continuous variable associated with column $r$; similarly, the second sum ranges over the columns $p$ of $P$, and $y_p$ denotes the integer variable associated with column $p$. This convention for summations will be used throughout the paper. Valid pairs are also known as {\em cut-generating pairs}, and inequality \eqref{psi pi ineq} is often called a {\em cut}.\medskip

When $\ell=0$ in \eqref{def mixed-int set}, i.e., when all variables are continuous, we obtain a set of the type
\begin{equation}\label{def continuous set}
	C_{S}(R) := \setcond{s \in \R_+^k}{ Rs \in S },
\end{equation}
where $k\ge1$. A function $\psi\colon\R^n \to \R$ is a \emph{valid function} for $S$ if  $\sum\psi(r)s_r \ge 1$  is an inequality separating 0 from $C_{S}(R)$ for every $k$ and $R$. Again we use the convention that the above sum is taken over the columns $r$ of $R$.
This model, here referred to as the {\em continuous model}, will be discussed in Section~\ref{sec:continuous-infinite-relax}.
\medskip

When $k=0$, i.e., when all variables are integer, sets of the form \eqref{def mixed-int set} become
\begin{equation}\label{def integer set}
	I_{S}(P) := \setcond{y \in \Z_+^{\ell}}{ Py \in S },
\end{equation}
where $\ell\ge1$.
A function $\pi\colon\R^n \to \R$ is an \emph{integer valid function} for $S$ if  $\sum\pi(p)y_p \ge 1$  is an inequality  separating 0 from $I_{S}(P)$ for every $\ell$ and $P$.
This model, here referred to as the {\em pure integer model}, will be discussed in Section~\ref{sec:pure}.
\medskip

When both $k$ and $\ell$ are positive, we refer to \eqref{def mixed-int set} as the {\em mixed integer model}; this will be discussed in Section~\ref{sec:mixed}.
\medskip

There is a natural partial order on the set of valid pairs, namely $(\psi', \pi') \leq (\psi, \pi)$ if and only if $\psi' \leq \psi$ and $\pi' \leq \pi$. Since $\{(s,y) \colon \sum\psi'(r)s_r + \sum\pi'(p)y_p \ge 1,s\ge 0, y\ge 0\}\subseteq \{(s,y) \colon \sum\psi(r)s_r + \sum\pi(p)y_p \ge 1,s\ge 0, y\ge 0\}$ whenever $(\psi', \pi') \leq (\psi, \pi)$,
all the cuts obtained from $(\psi, \pi)$ are dominated by those obtained from $(\psi', \pi')$. The minimal elements under this partial order are called {\em minimal valid pairs.} Similarly, one defines {\em minimal valid functions} $\psi$  and {\em minimal integer valid functions} $\pi$.  An application of Zorn's lemma (see e.g. \cite[Theorem 1.1]{basu-hildebrand-koeppe-molinaro:k+1-slope}) shows that every valid pair (resp., valid function, integer valid function) is dominated by a minimal valid pair (resp., minimal valid function, minimal integer valid function).
Thus one can concentrate on the minimal valid functions and pairs.

\begin{remark}
A natural question is whether cut generating functions are sufficient in the following sense: Given a fixed closed set $S \subseteq \R^n \setminus \{0\}$ and a fixed pair of matrices $R, P$, is the closed convex hull of $X_S(R,P)$ described by the intersection of all inequalities of the type~\eqref{psi pi ineq} when we consider all possible minimal valid pairs for $S$? The same question can be phrased for the continuous model~\eqref{def continuous set}, as well as for the pure integer model~\eqref{def integer set}.

This question, in its full generality, is not settled. For the continuous model~\eqref{def continuous set}, Conforti et al.~\cite[Example 6.1]{conforti2013cut} show that for a particular  set $S$  minimal valid functions do not suffice. However, if  $S$ is contained in  the conical hull of the columns of $R$,  Cornu\'ejols et al.~\cite{cornuejols2015sufficiency} prove that $C_S(R)$ is defined by the inequalities derived from cut-generating functions. Earlier Zambelli ~\cite{zambelli} showed this to be true  when $S = b + \Z^n$ for some $b \in \R^n \setminus \Z^n$.

For the pure integer model~\eqref{def integer set},  when $n=1$ and  $R$ is a rational matrix, an affirmative answer can be deduced from~\cite[Theorem 8.3]{bhk-survey} (this result is a restatement of results appearing in~\cite{infinite}).

\end{remark}

\paragraph{Notation}  Given a convex subset $K$ of $\R^n$, we denote with  $\dim(K)$, $\intr(K)$, $\relintr(K)$, $ \cl(K)$, $ \aff(K)$, $ \rec(K)$, $\lin(K)$ the  dimension, interior, relative interior, topological closure, affine hull, recession cone and lineality space of $K$.
These are standard notions in convex analysis, see e.g. \cite{lemarechal1996convex}.

 We denote with $B(x,\varepsilon)$ the closed ball of center $x$ and radius $\varepsilon$. Given $V\subseteq \R^n$, we indicate with $\conv(V)$ its convex hull and with $\langle V \rangle$ the linear space generated by $V$. Given a linear subspace $L$, we denote by $L^\perp$ the orthogonal complement of $L$ and by $\proj_L(\cdot)$ the orthogonal projection on $L$.

\section{Significance of the mixed integer set~\eqref{def mixed-int set}}\label{sec:significance}

The model \eqref{def mixed-int set} contains as special cases several classical optimization models. Some examples are illustrated below.

\begin{enumerate}

\item{\em Gomory's relaxation of IP and extensions}. The classical way in which model \eqref{def mixed-int set} arises is as follows; see~\cite{gom}. Let $x+Py=b$ be the system of equations that defines a (final) tableau of the linear-programming (LP) relaxation of a pure IP problem, whose feasible set is
    $ \setcond{(x,y) \in \Z_+^n \times \Z_+^\ell}{x + Py = b}$.

     When $b\in \Z^n$, the LP basic solution $x=b,\,y=0$ is a solution to the IP and is an optimal solution when the tableau is final.  When $b\not\in \Z^n$, a relaxation of the above set can be obtained by dropping the nonnegativity condition on $x$. Thus the feasible set of this relaxation can be expressed only in terms of $y$ as
     \begin{equation}\label{eq-corner}
     \setcond{y \in  \Z_+^\ell} { Py \in b+\Z^n}.
     \end{equation}
     Note that this fits the setting \eqref{def integer set}  where $S=b+\Z^n$, and $0\not\in S$ because $b\not\in \Z^n$.
     The convex hull of \eqref{eq-corner} is known as the \emph{corner polyhedron}.

    Of course, if $S=b-\Z_+^n$, i.e. the condition $x\ge 0$ is maintained, no relaxation occurs. This case was one of the main motivations to study sets of the type $S=(b+\Z^n)\cap Q$, where $Q$ is a rational polyhedron, see~\cite{bccz2,dey2010constrained,yildiz2015integer}.

    The above can be extended to the mixed integer case as follows. Let $x+Rs+ Py=b$ be the system of equations that defines a (final) tableau of the LP relaxation of a mixed integer program, where $s$ is the vector of nonnegative continuous variables, $y$ is the vector of nonnegative integer variables, and $x$ is the vector of basic nonnegative variables, which may be continuous or integer. However if the $x\ge 0$ constraint is relaxed, one may assume that $x$ is a vector of integer variables, as every continuous basic variable is defined by the corresponding equation of the tableau with no further restriction; hence these equations can now be dropped. The model in this case is
    \begin{equation}\label{eq-mixed-corner}
    \setcond{(s,y) \in  \R_+^k\times \Z_+^\ell} { Rs+Py \in b+\Z^n}.
    \end{equation}
    Andersen, Louveaux, Weismantel and Wolsey \cite{alww}, Borozan and Cornu\'ejols \cite{BorCor}, and Basu, Conforti, Cornu\'ejols and Zambelli \cite{bccz} studied the relaxation of the above model in which the integrality of the nonbasic variables is relaxed:  $\{s \in  \R_+^{k+\ell} :\, (R,P)s \in b+\Z^n\}$.
    This important special case, which fits  \eqref{def continuous set},  will be discussed in Section~\ref{sec:Zn+b}.
    Again, sets $S$ different from $b+\Z^n$ may be considered.


%

\item {\em Mixed integer (structured) convex programs.} Mixed integer programming with convex constraints is a powerful generalization of mixed integer linear programming that can model problems in applications with inherent nonlinearities~\cite{Atamturk_Narayanan_10,Atamturk_Narayanan_11,ceria1999convex,Cezik_Iyengar_05,stubbs1999branch}. The classical model here is  $$\setcond{(x,s,y) \in \R^n \times \R_+^k \times \Z_+^\ell }{ Rs + Py + x = b, \;\; x \in K\cap (\R^t\times \Z^{n-t})}$$ where $K$ is a convex set. A special case of this model is {\em mixed integer conic programming}, where $K$ is taken to be a closed, convex, pointed cone.
This framework is readily obtained from~\eqref{def mixed-int set} by setting $S=b - K\cap (\R^t\times \Z^{n-t})$.

\item

{\em Complementarity problems with integer constraints.}  In such problems, the feasible region consists of all integer points in a given polyhedron $Q= \{(x,y) \in \R^n_+ \times \R^\ell_+ \colon x+ Py = b\}$ that satisfy the  complementarity constraints
$ x_ix_j = 0,\, (i,j) \in E$ where $E$ is a subset of $\{1, \ldots, n\} \times \{1, \ldots, n\}$.
This can be modeled using~\eqref{def integer set} by setting $S = b - \{x \in \Z^n_+ \colon x_ix_j = 0,\, (i,j) \in E\}$.

%
\end{enumerate}

%

\section{The continuous model}\label{sec:continuous-infinite-relax}

Given a closed set $S\subseteq \R^n\setminus \{0\}$, we study valid functions $\psi\colon\R^n \to \R$ for the model $C_S(R)$, as defined in \eqref{def continuous set}. We characterize  the valid functions that are minimal. We will see that minimal valid functions for $S$ are naturally associated with maximal $S$-free convex sets. A closed, convex set $K\subseteq \R^n$ is \emph{$S$-free} if $\intr(K)\cap S=\emptyset$,  and an $S$-free convex set $K$ is \emph{maximal} if $K$ is not properly contained in any $S$-free convex set. With a straightforward application of Zorn's lemma, it can be shown that every $S$-free convex set is contained in a maximal one~\cite{conforti2013cut}.

In Section~\ref{sec:Zn+b} we will consider the case $S=b+\Z^n$ for some fixed $b\in\R^n\setminus\Z^n$. As discussed in Section~\ref{sec:significance}, this case is of particular importance in integer programming. We will then treat the more general case of an arbitrary closed set $S\subseteq\R^n\setminus\{0\}$ in Section~\ref{sec:general-S}.

\subsection{The case  $S=b+\Z^n$}\label{sec:Zn+b}

Here we assume $S=b+\Z^n$ for a fixed $b\in\R^n\setminus\Z^n$, hence $0\notin S$.

We recall some definitions from convex analysis. A function $g\colon \R^n \to \R$ is {\em positively homogeneous} if $g(\lambda r)=\lambda g(r)$ for every $r\in \R^n$ and every $\lambda> 0$. \index{positively homogeneous} Note that if $g$ is positively homogeneous, then $g(0)=0$. Indeed, for any $\lambda>0$, we have that  $g(0)=g(\lambda 0)=\lambda g(0)$, which implies that $g(0)=0$. A function $g\colon \R^n \to \mathbb{R}$ is {\em subadditive} if $g(r^1)+g(r^2) \geq g(r^1+r^2)$  for all $r^1,r^2 \in \R^n$.  \index{subadditive} The function $g$ is {\em sublinear} if it is both subadditive and positively homogeneous. Note that since sublinear functions are convex, they are continuous in the interior of their domain. The following lemma appears first in \cite{BorCor}.


\begin{lemma} \label{Le:Cont-oldProperties}
Assume $S=b+\Z^n$ for some $b\notin\Z^n$, and let $\psi\colon\R^n\to \R$ be a minimal valid function for $S$. Then $\psi$ is sublinear and nonnegative.
\end{lemma}

\begin{proof} We first note that $\psi (0) \geq 0$. Indeed, consider any point $\bar s\in C_S(R)$ for some $n\times k$ matrix $R$ containing the 0-column.
Let $\tilde s = \bar s$ except for the component $\tilde s_0$, which is set to an arbitrarily large value $k$. Since $\tilde s \in C_S(R)$ and $\psi$ is valid, we have that  $\sum \psi(r) \tilde s_r+\psi(0)k\geq 1$.
For this inequality to hold for all $k>0$, we must have  $\psi (0) \geq 0$.\medskip

\noindent (a) {\em $\psi$ is sublinear.} We first prove that $\psi$ is subadditive. When $r^1=0$ or $r^2=0$, inequality $\psi(r^1) + \psi (r^2) \geq \psi(r^1+r^2)$ follows from $\psi (0) \geq 0$. Assume that for  $r^1,r^2\neq 0$,  $\psi(r^1) + \psi (r^2) < \psi(r^1+r^2)$.  Set $\psi'(r^1+r^2)=\psi(r^1) + \psi (r^2)$ and  $\psi'(r)=\psi(r)$ for $r\ne r^1+r^2$. Then $\psi'\le\psi$, $\psi'\ne\psi$. We show that  $\psi'$ is  a valid function, a contradiction to the minimality of $\psi$.

Consider any $\bar s  \in C_S(R)$ for some matrix $R$.  We assume, without loss of generality, that $r^1$, $r^2$ and $ r^1 + r^2$ are columns of $R$ (otherwise, simply add the missing vectors as columns and put a value of 0 for the corresponding component of $\bar s$).
Define $\tilde s$ as follows:
$$\tilde s_r := \left\{ \begin{array}{ll} \bar s_{}+\bar s_{r^1 + r^2} &
\mbox{if } r = r^1 \\ \bar s_{r}+\bar s_{r^1 + r^2} &
\mbox{if } r = r^2 \\ 0 & \mbox{if } r = r^1+r^2 \\
\bar s_r & \mbox{otherwise.} \end{array} \right.$$
Note that $\tilde s \geq 0$ and $R\tilde s = R\bar s  \in S$, thus $\tilde s \in C_S(R)$. Using the definitions of $\psi'$ and $\tilde s$, it is easy to verify that
$\sum \psi'(r) \bar s_r=\sum \psi (r) \tilde s_{r} \geq 1,$ where the last inequality follows from the facts that $\psi$ is valid and $\tilde s\in C_S(R)$. This shows that $\psi'$ is valid. \medskip

We next show that $\psi$ is positively homogeneous. Suppose there exists $\tilde r\in\R^n$ and $\lambda> 0$ such that  $\psi(\lambda\tilde r)\neq \lambda\psi(\tilde r)$. Without loss of generality we may assume that $\psi(\lambda\tilde r)< \lambda\psi(\tilde r)$. Define a function $\psi'$ by
$\psi'(\tilde r) :=  \lambda^{-1}\psi(\lambda\tilde r)$, $\psi'(r):=\psi(r)$ for all $r \not= \tilde r$.
 It is easy to see that $\psi'$ is valid, contradicting the fact that $\psi$ is minimal.
Therefore $\psi$ is positively homogeneous.
\medskip

\noindent (b) {\em $\psi$ is nonnegative.} Suppose $\psi (\tilde r) < 0$ for some $\tilde r\in\Q^n$. Let $q\in\Z_+$ be such that $q\tilde r\in \Z^n$ and let $\bar s\in C_S(R)$, where $\tilde r$ is a column of $R$. Let $\tilde s$ be defined by $\tilde s_{\tilde r} := \bar s_{\tilde r} + Mq$
where $M$ is a positive integer, and $\tilde s_r := \bar s_r$ for $r \ne \tilde r$.
Then $\tilde s\in C_S(R)$ and $\sum \psi (r) \tilde s_r = \sum \psi (r) \bar s_r + \psi (\tilde r) Mq$. Since $\psi(\tilde r)Mq<0$ and $M$ is any positive integer, this sum can be made smaller than 1, a contradiction to the validity of $\psi$.

Since $\psi$ is sublinear, $\psi$ is convex and therefore  continuous. Thus, as $\psi$ is nonnegative over $\Q^n$ and $\Q^n$ is dense in $\R^n$, $\psi$ is nonnegative over $\R^n$.
\end{proof}

Let $K \subseteq \R^n$ be a closed convex set with the origin in its interior.
A standard concept in convex analysis~\cite{lemarechal1996convex,Rock} is that of {\em gauge}, which is the function $\gamma_K$  defined by
$$\gamma_K(r) :=\inf\left\{t>0\st  \frac{r}{t}\in K\right\}\quad \mbox{ for all } r\in \R^n.$$
Since the origin is in the interior of $K$, $\gamma_K(r)<+\infty$ for all $r\in \R^n$.
Furthermore $\gamma_K(r) \leq 1$ if and only if $r \in K$, and $\intr(K)=\{r\in \R^n\st \gamma_K(r) < 1\}$: since $\gamma_K$ is a continuous function, $\gamma_K(\bar r)<1$ implies $\gamma_K(r)< 1$  for every $r$ close to $\bar r$, and since $\gamma_K$ is positively homogeneous, we have that $\gamma_K(\bar r)=1$ implies $\gamma_K( r)> 1$ if $r=(1+\varepsilon)\bar r$ for $\varepsilon>0$.

The following lemma is standard in convex analysis, see for instance  \cite{lemarechal1996convex}.

\begin{lemma} \label{lemma:gauge->sub} Given a closed convex set $K\subseteq\R^n$ with the origin in its interior, the gauge $\gamma_K$ is a nonnegative sublinear function.\\ Conversely, given a function $\gamma\colon\; \R^n\to \R$ which is nonnegative and sublinear, let $$K_{\gamma}:=\{x\in \R^n\st \gamma(x)\le 1\}.$$ Then $K_{\gamma}$ is a closed convex set with the origin in its interior, and $\gamma$ is the gauge of $K_{\gamma}$.
\end{lemma}


\begin{lemma}\label{lemma:psi-gauge1}
Assume $S=b+\Z^n$ for some $b\notin\Z^n$. Let $K\subseteq\R^n$ be a closed convex set with $0\in \intr(K)$ and
let $\psi$ be the gauge of $K$. Then $\psi$ is a valid function for $S$ if and only if $K$ is $S$-free.
\end{lemma}
\begin{proof} By Lemma \ref{lemma:gauge->sub}, $\psi$ is sublinear.
We prove the ``if'' part. Assume that $K$ is $S$-free and consider $s\in C_S(R)$ for some matrix $R$. That is, $\sum rs_r= b+x$, where $x\in \Z^n$ and the sum ranges  over the columns $r$ of $R$. Then
$$\textstyle \sum\psi(r)s_r=\sum\psi(rs_r)\ge \psi(\sum rs_r)=\psi(b+x)\ge 1,$$
where the first equality follows by positive homogeneity of $\psi$, the first inequality by subadditivity, and the last from the fact  that $\psi$ is the gauge of $K$ and $b+x\notin  \intr(K)$ because $x\in \Z^n$ and $K$ is $S$-free.

For the ``only if'' part, assume $b+x\in \intr(K)$, with $x\in\Z^n$. Let $R$ be the $n\times 1$ matrix $b+x$. Then the point defined by $s_{b+x}=1$ is in $C_S(R)$, and $\psi(b+x)<1$ because $\psi$ is the gauge of $K$ and $b+x\in \intr(K)$. Thus $\psi$ is not a valid function for $S$.
\end{proof}

The following theorem (see \cite{bccz,BorCor}) shows the correspondence between minimal valid functions for $S$ and maximal $S$-free convex sets, when $S=b+\Z^n$.

\begin{theorem}\label{thm:psi-B} Assume $S=b+\Z^n$ for some $b\notin\Z^n$. A function $\psi\colon\R^n\to \R$ is a minimal valid function for $S$ if and only if there exists some maximal $S$-free convex set $K\subseteq\R^n$ such that $0\in \intr(K)$ and $\psi$ is the gauge of $K$.
\end{theorem}
\begin{proof} Assume that $\psi$ is a minimal valid function. By Lemma \ref{Le:Cont-oldProperties}, $\psi$ is a nonnegative sublinear function, and by Lemma \ref{lemma:gauge->sub}, $\psi$ is the gauge of a closed convex set $K$ such that $0\in \intr(K)$. Since $\psi$ is a valid function for $S$, by Lemma \ref{lemma:psi-gauge1}, $K$ is an $S$-free convex set.
We prove that $K$ is a maximal $S$-free convex set. Suppose not, and let $K'$ be an $S$-free convex set properly containing $K$. Let $\psi'$ be the gauge of $K'$. By Lemma \ref{lemma:psi-gauge1}, $\psi'$ is a valid function, and since $K\subsetneq K'$, we have that $\psi'\leq\psi$ and $\psi'\neq \psi$.  This contradicts the minimality of $\psi$. The converse is straightforward.
\end{proof}

When $S=b+\Z^n$, in view of Theorem \ref{thm:psi-B} characterizing minimal valid functions amounts to characterizing maximal $S$-free convex sets containing $0$ in their interior. We will see that every maximal $S$-free convex set is a polyhedron. Therefore, if a maximal $S$-free convex set contains 0 in its interior, then it is a polyhedron of the form $K=\{x\in \R^n\st a_ix\le 1 ,\, i\in I\}$ for some finite set $I$.  This turns out to be very useful, as it can be employed to obtain an explicit formula for the computation of the minimal valid function associated with $K$, i.e., the gauge of $K$.
The formula is stated in the following theorem.

\begin{theorem}\label{TH-gaugePolyhedron}
Assume $S=b+\Z^n$ for some $b\notin\Z^n$.
Then every maximal $S$-free convex set is a polyhedron.
Moreover, if a maximal $S$-free polyhedron $K$ with $0\in\intr(K)$ is given by $K=\{x\in \R^n\st a_ix\le 1,\,\forall i\in I\}$ for some finite set $I$, then the gauge $\psi$ of $K$ is
\begin{equation}\label{eq-gauge-poly}\psi(r)=\max_{i\in I}a_ir.\end{equation}
\end{theorem}

Since a set $K$ is a (maximal) $S$-free convex set if and only if $K-b$ is a (maximal)  $\Z^n$-free convex set, the proof of Theorem~\ref{TH-gaugePolyhedron} requires the characterization of maximal $\Z^n$-free convex sets and is postponed until the end of subsection~\ref{s:mlfc}.

\subsubsection{Maximal lattice-free convex sets}\label{s:mlfc}

We characterize the structure of maximal $\Z^n$-free sets in this section in Theorem~\ref{TH-maximal} and then derive Theorem~\ref{TH-gaugePolyhedron} as a consequence.

Our treatment uses basic facts about lattices. A \emph{lattice}  of dimension $t$  is a set of the type $\{x\in \R^n\st x=\lambda_1a_1+\dots+\lambda_ta_t;\,\lambda_1,\dots,\lambda_t\in \Z\}$, where $a_1,\dots,a_t$ are linearly independent vectors in $\R^n$. It follows from this definition that $\Z^n$ is a lattice. We call a $\Z^n$-free convex set \emph{lattice-free}.
We refer to Chapter VII in the book of Barvinok \cite{barvinok} for an introduction to lattice theory.

A convex set $C$ is centrally symmetric with center $p$ if $x \in C$ implies $2p-x \in C$. We will sometimes simply say $C$ is centrally symmetric, if there exists $p \in C$ such that $C$ is centrally symmetric with center $p$. 

\begin{theorem}[Minkowski's convex body theorem (see, e.g., \cite{barvinok})]\label{th:mink}
Let $C\subseteq\R^n$ be a centrally symmetric convex set with center $0$. If $\vol(C)>2^n$, then $C$ contains a nonzero integer point. Moreover, if $C$ is compact, the condition can be relaxed to $\vol(C)\ge2^n$.
\end{theorem}

A subspace $H\subseteq\R^n$ is  a lattice subspace if $\dim(H\cap \Z^n)=\dim(H)$.  That is, if $H$ can be generated by an integral basis. Equivalently, $H=\{x\in \R^n: Ax=0\}$ for some  $(n-\dim(H))\times n$ integral  matrix $A$ of full row-rank.

  Given a linear subspace $L\subseteq\R^n$, there exists a unique  minimal lattice subspace containing $L$. It is the intersection of all lattice subspaces containing $L$.

\begin{lemma}\label{lem:closed}
Let $H\subseteq\R^n$ be a lattice subspace. Then $H+\Z^n$ is a closed set.
\end{lemma}

\begin{proof}
Since $H$ is a lattice subspace, by applying a suitable unimodular transformation we can assume that $H=\{x\in\R^n \colon x_1=\dots=x_k=0\}$ for some $k\in\{0,\dots,n\}$.
Then $H+\Z^n=\{x\in\R^n \colon x_1,\dots,x_k\in\Z\}$, which is a closed set.
\end{proof}

\begin{prop}\label{prop:proj}
Let $L\subseteq\R^n$ be a linear subspace and let $H$ be the minimal lattice subspace containing $L$.
Then
\[\cl(\proj_{L^\perp}(\Z^n))=(H+\Z^n)\cap L^\perp=(H\cap L^\perp)+\Lambda\]
for some lattice $\Lambda\subseteq L^\perp$ such that $\dim(H\cap L^\perp)+\dim(\Lambda)=\dim(L^\perp)$.
\end{prop}

\begin{proof}
Note that $\proj_{L^\perp}(\Z^n)=(L+\Z^n)\cap L^\perp$. By Lemma~\ref{lem:closed}, $(H+\Z^n)\cap L^\perp$ is a closed set containing $(L+\Z^n)\cap L^\perp$. It follows that $\cl(\proj_{L^\perp}(\Z^n))=\cl((L+\Z^n)\cap L^\perp)\subseteq(H+\Z^n)\cap L^\perp$.

To show the reverse inclusion, we first assume $H=\R^n$. In this case, $(H+\Z^n)\cap L^\perp=L^\perp$, thus we have to prove that for every $x\in L^\perp$ and $\epsilon>0$, the ball $B(x,\epsilon)$ intersects $\proj_{L^\perp}(\Z^n)$.

The proof is by (reverse) induction on $k:=\dim(L)$. The case $k=n$ is trivial. Now assume $k<n$.
Fix $x\in L^\perp$ and $\epsilon>0$ (with $\epsilon<1$ without loss of generality), and assume by contradiction that no point in $\proj_{L^\perp}(\Z^n)$ belongs to $B(x,\epsilon)$.
We claim that $\proj_{L^\perp}(\Z^n)$ contains a nonzero vector $w$ such that $\|w\|\le\epsilon/2$.
To see this, define $\Lambda'=L\cap\Z^n$. Since $L$ is not a lattice subspace (as $L\subsetneq H$), $\Lambda'$ is a lattice of dimension smaller than $k$.
Then there exists a nonzero vector $v\in L\cap \langle\Lambda'\rangle^\perp$.  Define $B'=B(0,\epsilon/2)\cap \langle v\rangle^\perp$.
Let $C$ be the centrally symmetric cylinder $C=B'+[-\lambda v,\lambda v]$, where $\lambda>0$. For $\lambda$ large enough, $\vol(C)>2^n$, thus, by Minkowski's convex body theorem, $C$ contains a nonzero integer point $z$. Note that $z\notin L$: otherwise, we would have $z\in C\cap L\cap\Z^n=C\cap \Lambda'\subseteq B'$; but $B'$ contains no integer point other than the origin, as $\epsilon<1$. Therefore $z\notin L$, which implies that its projection $w$ onto $L^\perp$ is not the origin. Note that $\|w\|\le\epsilon/2$, as claimed.

We now claim that the unbounded cylinder $C'=B(x,\epsilon/2)+\langle w\rangle$ contains no point from $\proj_{L^\perp}(\Z^n)$.
Assume to the contrary that there exists $y\in\proj_{L^\perp}(\Z^n)$ such that $y=x'+\mu w$, where $x'\in B(x,\epsilon/2)$ and $\mu\in\R$.
Then the point $y'=y-\floor{\mu}w=x'+(\mu-\floor\mu)w$ would also belong to $\proj_{L^\perp}(\Z^n)$; however
\[\|y'-x\|=\|y'-x'\|+\|x'-x\|\le\|w\|+\epsilon/2\le\epsilon,\]
thus $y'\in B(x,\epsilon)$, which is a contradiction, as $B(x,\epsilon)$ contains no point from $\proj_{L^\perp}(\Z^n)$.
Therefore $C'$ contains no point from $\proj_{L^\perp}(\Z^n)$, as claimed.

Now, if $k=n-1$ (i.e., $L^\perp$ is a line), by choosing $\epsilon$ arbitrarily small the norm of $w$ can be made arbitrarily small, and we conclude that $\cl(\proj_{L^\perp}(\Z^n))=L^\perp$.
So we assume $k<n-1$.
Define $L'=\langle L \cup \{w\}\rangle$. Since the minimal lattice subspace containing $L'$ is $\R^n$, by induction $\cl(\proj_{(L')^\perp}(\Z^n))=(L')^\perp$. However, the projection of $C'$ onto $(L')^\perp$ is a ball in $(L')^\perp$ that contains no point from $\proj_{(L')^\perp}(\Z^n)$, a contradiction.

This concludes the proof for the case $H=\R^n$. If $H$ is subspace of dimension $n'<n$, modulo a unimodular transformation we can assume that $H=\R^{n'}\times\{0\}^{n-n'}$. We can apply the result with respect to the ambient space $H$, which is equivalent to $\R^{n'}$. We then have $\cl(\proj_{L^\perp}(H\cap\Z^n)=H\cap L^\perp$. Similarly, for every $a\in\Z^n$ we have $\cl(\proj_{L^\perp}((H+a)\cap\Z^n))=(H+a)\cap L^\perp$. Then
$$\cl(\proj_{L^\perp}(\Z^n))\supseteq\bigcup_{a\in\Z^n}\cl(\proj_{L^\perp}((H+a)\cap\Z^n))=\bigcup_{a\in\Z^n}(H+a)\cap L^\perp=(H+\Z^n)\cap L^\perp.$$

It remains to show that $(H+\Z^n)\cap L^\perp=(H\cap L^\perp)+\Lambda$ for some lattice $\Lambda\subseteq L^\perp$ such that $\dim(H\cap L^\perp)+\dim(\Lambda)=\dim(L^\perp)$. Define $\Lambda=\proj_{H^\perp}(\Z^n)$. Since $H^\perp$ is a lattice subspace, $\Lambda$ is a lattice.
Moreover, since $L\subseteq H$, we have that $\Lambda\subseteq H^\perp\subseteq L^\perp$. By writing
\[(H+\Z^n)\cap L^\perp=\proj_H((H+\Z^n)\cap L^\perp)+\proj_{H^\perp}((H+\Z^n)\cap L^\perp),\]
and observing that $\proj_H((H+\Z^n)\cap L^\perp)=(H\cap L^\perp)$ and $\proj_{H^\perp}((H+\Z^n)\cap L^\perp)=\Lambda$, one concludes that $(H+\Z^n)\cap L^\perp=(H\cap L^\perp)+\Lambda$.
\end{proof}

We illustrate the above proposition in the case $n=3$ with $L$ being a line. If $H=L$, then $\cl(\proj_{L^\perp}(\Z^n))=\proj_{L^\perp}(\Z^n)$ is a lattice; if $\dim(H)=2$, then $\cl(\proj_{L^\perp}(\Z^n))$ is the union of discrete shifts of $H\cap L^\perp$; finally if $H=\R^3$,  $\cl(\proj_{L^\perp}(\Z^n))=L^\perp$. See Figure~\ref{fig:proj}.

\begin{figure}
\begin{center}
\begin{pspicture}(15,6.5)(0,1.5)
\pspolygon[linestyle=none,fillcolor=lightgray,fillstyle=solid](10,3)(13.5,3)(15,4.5)(11.5,4.5)
\multiput(0,0)(5,0){3}{
\pspolygon(0,3)(3.5,3)(5,4.5)(1.5,4.5)
\put(2.5,5.5){$L$}
\put(0.05,3.5){$L^\perp$}
}
\multiput(0,0)(0.5,0.5){3}{\multiput(0.6,3.25)(0.65,0){5}{\color{black!70!white}\circle*{.1}}}
\multiput(0,0)(5,0){3}{
\psline(2.4,3.75)(2.4,6)
\psline[linestyle=dotted](2.4,3.75)(2.4,3)
\psline(2.4,3)(2.4,2)
}
\psline(9,3.5)(9,2.25)(6.05,1.58)(6.05,5.8)(9,6.43)(9,4.5)
\psline[linestyle=dotted](9,3.5)(9,4.5)
\psline(7.4,3.75)(7.4,3.68)
\psline[linecolor=black!70!white](6.05,3)(9,4.5)
\psline[linecolor=black!70!white](7.15,3)(9.9,4.4)
\psline[linecolor=black!70!white](8.25,3)(8.72,3.215)
\psline[linecolor=black!70!white](5.07,3.06)(6.05,3.55)
\psline[linecolor=black!70!white,linestyle=dashed](6.15,3.6)(7.9,4.5)
\psline[linecolor=black!70!white,linestyle=dashed](6.2,4.2)(6.8,4.5)
\put(6.2,5.3){$H$}
\rput{10}(11.9,5.9){\psellipse[linewidth=.4pt](0.5,0)(.5,.2)}
\rput{10}(11.9,1.9){\psellipticarc[linewidth=.4pt](0.5,0)(.5,.2){175}{0}}
\rput{10}(11.9,1.9){\psellipticarc[linestyle=dotted](0.5,0)(.5,.2){0}{175}}
\rput{10}(11.9,3.65){\psellipticarc[linewidth=.4pt](0.5,0)(.5,.2){175}{0}}
\rput{10}(11.9,3.65){\psellipticarc[linestyle=dotted](0.5,0)(.5,.2){0}{175}}
\psline[linewidth=.4pt](12.88,6.07)(12.88,2.07)
\psline[linewidth=.4pt](11.905,5.92)(11.905,1.92)
\end{pspicture}
\end{center}
\caption{Illustration of Proposition~\ref{prop:proj} for $n=3$ with $L$ being a line. In the first picture, $H=L$ and the orthogonal projection of $\Z^3$ onto $L^\perp$ is a 2-dimensional lattice. In the second picture, $H$ is a plane, and the set $\cl(\proj_{L^\perp}(\Z^3))$ is the union of discrete shifts of the line $H\cap L^\perp$. In the third picture, $H=\R^3$; the cylinder indicates that there are integer points in any neighborhood of $L$, and the projection of $\Z^3$ is dense in $L^\perp$ (i.e., its closure is $L^\perp$).}
\label{fig:proj}
\end{figure}
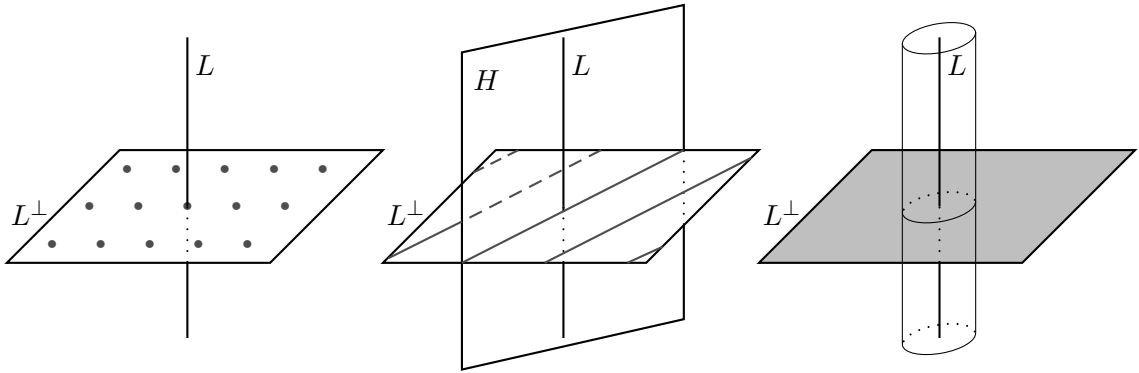

\begin{theorem}\label{TH-maximal}
A set $K\subseteq\R^n$ is a maximal lattice-free convex set if and only if it satisfies one of the following conditions:
\begin{enumerate}[\upshape(a)]
\item
$K=a+L$, where $a\in\R^n$ and $L$ is a subspace of dimension $n-1$ that is not a lattice subspace.
\item
$K$ is an $n$-dimensional polyhedron of the form $K=Q+L$, where $L$ is a lattice subspace of dimension $r$, with $r<n$, $Q$ is a polytope of dimension $n-r$, and the relative interior of every facet of $K$ contains an integer point.
\end{enumerate}
\end{theorem}

\begin{proof}
We first prove the ``if'' direction. Assume that (a) holds; we show that $K$ is a maximal lattice-free convex set.
Since the minimal lattice subspace containing $L$ is $\R^n$, Proposition~\ref{prop:proj} implies $\cl(\proj_{L^\perp}(\Z^n))=L^\perp$.
Then for every $\epsilon>0$ and $a\in\R^n$, there is a point in $\Z^n\setminus(a+L)$ at distance at most $\epsilon$ from $a+L$.
Suppose that $K'$ is a closed convex set that strictly contains $K$.
Since $K'$ is closed, it must contains a set of the form $[a,a+v]+L$ for some $v\in L^\perp\setminus\{0\}$.
Then $\intr(K')$ contains an integer point, a contradiction to the assumption that $K'$ is lattice-free.
It follows that no lattice-free convex set strictly contains $K$ and thus $K$ is a maximal lattice-free convex set.\smallskip

Let now $K$ satisfy (b); we prove that $K$ is a maximal lattice-free convex set.
If $K'$ is a lattice-free convex set strictly containing $K$, then there is a facet $F$ of $K$ such that $\relint(F)\subseteq\intr(K')$.
Since $\relint(F)$ contains an integer point, this point is in $\intr(K')$, a contradiction to the assumption that $K'$ is lattice-free.
It follows that $K$ is a maximal lattice-free convex set.
\smallskip

We now prove the ``only if'' direction. We first assume that $K$ is a maximal lattice-free convex set with $\dim(K)<n$; we show that (a) holds.
By maximality, $K$ is a hyperplane, hence $K=a+L$ for some $a\in\R^n$ and some linear subspace $L$ of dimension $n-1$.
If $L$ is a lattice subspace, then $L=\{x\in\R^n\colon cx=0\}$ for some $c\in\Z^n$.
Define $\alpha=ca$ and $K'=\{x\in\R^n\colon\floor\alpha\le cx \le\floor\alpha +1\}$.
Then $K'$ is a lattice-free convex set that strictly contains $K$, a contradiction to the maximality of $K$.
It follows that $L$ is not a lattice-subspace and (a) holds.
\smallskip

We finally show that if $K$ is an $n$-dimensional lattice-free convex set then (b) is satisfied.\smallskip

\noindent
{\sc Claim.} $\lin(K)=\rec(K)$.
\begin{cpf}
We assume $\rec(K)\ne\{0\}$, otherwise the statement holds trivially.
Define $K'=K-\rec(K)$ and assume that there is an integer point $z\in\intr(K')$.
Choose $\epsilon>0$ such that $B(z,\epsilon)\subseteq\intr(K')$ and let $v\in\relint(\rec(K))$.
Then by Theorem~\ref{th:mink} the set $B(z,\epsilon)+\R_+v$ contains integer points arbitrarily far from $z$.
On the other hand, by the choice of $v$, every point of the form $x+\lambda v$ with $x\in B(z,\epsilon)$ and $\lambda$ large enough belongs to $\intr(K)$.
This contradicts the fact that $K$ is lattice-free.
Therefore $\intr(K')$ contains no integer point. Since $K$ is maximally lattice-free and $K\subseteq K'$, it follows that $K=K'$, i.e., $\lin(K)=\rec(K)$.
\end{cpf}

\noindent
{\sc Claim.} $\lin(K)$ is a lattice subspace.
\begin{cpf}
Define $L=\lin(K)$ and let $H$ be the minimal lattice subspace containing $L$.
If $L$ is not a lattice subspace, then $L\subsetneq H$, and therefore there exists a nonzero vector $v\in H\cap L^\perp$.
Define $K'=K+\langle v\rangle$ and assume that $\intr(K')$ contains an integer point $z$. Then $z=x+\lambda v$ for some $x\in\intr(K)$ and $\lambda\in\R$.
Since $x+L\subseteq H+\Z^n$, by Proposition~\ref{prop:proj} we have $(x+L)\cap L^\perp\subseteq\cl(\proj_{L^\perp}(\Z^n))$.
This implies that there are integer points that are arbitrarily close to $x+L$. Since $x+L\subseteq\intr(K)$, this contradicts the fact that $K$ is lattice-free.
We conclude that $K'$ is lattice-free, which is a contradiction to the maximality of $K$.
\end{cpf}

By the claims, $K=Q+L$ where $L$ is a lattice subspace and $Q=K\cap L^\perp=\proj_{L^\perp}(K)$.
Note that $Q$ is a bounded set, $0\le\dim(L)\le n-1$ and $\dim(Q)=n-\dim (L)$.

It remains to prove that $Q$ is a polytope and the relative interior of every facet of $K$ contains an integer point.  Since $L$ is a lattice subspace, by Proposition \ref{prop:proj},  $\proj_{L^\perp}(\Z^n)$ is a lattice $\Lambda$. Since $K$ is a maximal lattice-free convex set  and $K=Q+L$, then  $\intr(Q)\cap \Lambda=\emptyset$ and $Q$ is a maximal $\Lambda$-free convex set (in the space $L^\perp$).  Since $Q\subseteq L^\perp$ is a bounded set, then $Q\subseteq B\subseteq L^\perp$, where $B$ is a box. Let $B\cap \Lambda=\{z_1,\dots, z_k\}$. Since $\intr(Q)\cap \Lambda=\emptyset$, for every $z_i$ there exists an half-space $H_i$ containing $Q$ and having $z_i$ on the boundary. Then $Q\subseteq B\cap H_1,\dots\cap H_k$, and in fact $Q= B\cap H_1,\dots\cap H_k$ by maximality of $Q$.  This shows that $Q$ is a polytope.

Assume that a facet $F$ of $Q$ is such that $\relint(F)\cap \Lambda=\emptyset$. (Equivalently, the relative interior of the facet $F+ L$ of $K$ does not contain an integer point.) Since $F$ is a polytope and $\Lambda$ is a lattice, every point in $\Lambda$ is at distance at least $\varepsilon$ from $\relint(F)$, for some $\varepsilon>0$. Let $Q'$ be obtained from $Q$ by relaxing $F$ by $\varepsilon$. By construction,  $\relint(Q')\cap \Lambda=\emptyset$ and $Q'\supsetneq Q$. Let $K'=Q'+ L$. Then $K'$ is a lattice-free convex set and $K'\supsetneq K$, a contradiction to the maximality of $K$.
\end{proof}

We remark that the above theorem also holds for maximal $(b+\Z^n)$-free convex sets, except that the last part of condition (b) becomes: ``the relative interior of every facet of $K$ contains a point in $b+\Z^n$''.

The original proof of Theorem~\ref{TH-maximal} of Basu et al.\@ \cite{bccz} uses the Dirichlet approximation theorem. The subsequent (short) proof of Averkov  \cite{MR3027668}  uses the convex body theorem of Minkowski (Theorem~\ref{th:mink}) which implies the theorem of Dirichlet.

As shown in Theorem~\ref{TH-maximal}, maximal lattice-free convex sets are polyhedra. The following result bounds the number of facets of these polyhedra.

\begin{theorem} [Doignon~\cite{Doignon1973}, Bell~\cite{bell1977theorem}, Scarf~\cite{scarf1977observation}]\label{thm:DBS}
Let $K\subseteq \R^n$ be a maximal lattice-free convex set. Then $K$ is a polyhedron with at most $2^n$ facets.
\end{theorem}
\begin{proof} By Theorem \ref{TH-maximal}, every facet of $K$ contains a point in $\Z^n$ in its relative interior. If $K$ has more than $2^n$ facets, then $K$ has two facets whose relative interiors contain points say $z_1,\;z_2\in \Z^n$ that are congruent modulo 2, i.e., $z_1 - z_2$ has even components. But then $z=\frac{1}{2}z_1+\frac{1}{2}z_2$ is an integral point in $\intr(K)$, a contradiction.
\end{proof}

We can now prove the correctness of formula \eqref{eq-gauge-poly} to compute the gauge $\psi$ of $K$.\medskip

\noindent \emph{Proof of Theorem \ref{TH-gaugePolyhedron}.}
Recall that $K=\{x\in\R^n\colon a_ix\le1,\,i\in I\}$ is a maximal ($b+\Z^n$)-free convex set with $0\in\intr(K)$. Then $\dim(K)=n$ and $K$ satisfies (b) of Theorem \ref{TH-maximal}, hence $\dim(\rec(K))<n$. Since $\rec(K)=\{x\in \R^n\st a_ix\le0,\,i\in I\}$ and $0\in K$, then $\rec(K)\subseteq K$.

Fix $r\in\R^n$ and let $k$ be an index in $I$ such that $a_kr=\max_{i\in I}a_ir$.
Assume that $a_kr<0$. Then there exists $\epsilon>0$ such that $a_kr'<0$ for every $r'\in B(r,\epsilon)$, and therefore $\dim(\rec(K))=n$, a contradiction.

It follows that $a_kr\ge0$. If $a_kr=0$ then $r$ belongs to $\rec(K)$. In this case $\psi(r)=0$ and \eqref{eq-gauge-poly} holds.
So we assume $a_kr>0$. Define $t=a_kr$. Then, for $i\in I$, $a_i(r/t)\le 1$, and thus $\psi(r)\le t$.
On the other hand, if $t'<t$ then $a_k(r/t')>1$. This proves that $\psi(r)=t$. \qed

\subsection{General sets $S$}\label{sec:general-S}

In this section we show how the theory developed in Section~\ref{sec:Zn+b} for the case $S=b+\Z^n$ extends to more general closed sets $S\subseteq \R^n\setminus \{0\}$.

\subsubsection{Minimal valid functions}

Conforti, Cornu\'ejols, Daniilidis, Lemar\'echal and Malick \cite{conforti2013cut} studied the link between minimal valid functions and $S$-free convex sets, independently from the structure of maximal $S$-free convex sets. We summarize some of their results here.

\begin{theorem}\label{th-minimal-sublinear}  Given a closed set $S\subseteq \R^n\setminus \{0\}$, let $\psi\colon\R^n\to \R$ be a  valid function for $S$, and let $\psi'$ be defined as
$$\psi'(\hat r)=\inf \left\{\sum \psi(r)s_{r}\st  \sum r s_{r}=\hat r,s_{r}\ge 0\right\} \mbox{ for every $\hat r\in\R^n$}.$$
Then $\psi'$ is a valid function $\R^n\to \R$ which is sublinear.
\end{theorem}
In the above theorem, the summations are taken over all finite subsets of $\R^n$. We also stress that the above theorem implies that $\psi'$ cannot take the value $-\infty$. If $\psi$ and $\psi'$ are as in Theorem \ref{th-minimal-sublinear}, then $\psi'\le \psi$ by definition. Therefore to characterize minimal valid functions, one can concentrate on sublinear functions.
However, unlike the case $S=b+\Z^n$ (see Lemma \ref{Le:Cont-oldProperties}), a minimal valid function can take negative values.

Given a sublinear function $\rho$, let
 \begin{equation}\label{eq-describe}
V_\rho:=\bigl\{r\in\R^n:\rho(r)\leqslant 1\bigr\}\,.
\end{equation}
Then $V_\rho$ is a closed convex set and $0\in \intr(V_\rho)$. Conversely, given a closed convex set $V$ with $0\in \intr(V)$, a sublinear function $\rho$ such that $V=V_\rho$ is a \emph{representation}  of $V$.

\begin{theorem}\label{th-link}
Let $S\subseteq\R^n\setminus\{0\}$ be a closed set, let $\rho$ be a  sublinear function, and let  $V_\rho$ be defined as in  \eqref{eq-describe}.
Then $\rho$ is a valid function for $S$ if and only $V_\rho$ is $S$-free.
\end{theorem}

In view of Theorems \ref{th-minimal-sublinear} and \ref{th-link}, to characterize minimal valid functions for $S$ one has to study representations of $S$-free convex sets, which are in general not unique.
 However, these representations satisfy the following:

 \begin{theorem}\label{th-representations} Let $V\subseteq \R^n$ be a closed convex set with $0\in \intr(V)$ and let $\rho$ be a representation of $V$.
 Then $$\rho(r)\le 0\Longleftrightarrow r\in \rec(V), \mbox{ and } \rho(r)< 0\Longrightarrow r\in \intr(\rec(V)).$$
 Furthermore all representations of $V$ coincide in $V\setminus \intr(\rec(V))$.
\end{theorem}

Theorem \ref{TH-maximal} shows that if a polyhedron $K=\{x\in \R^n\st a_ix\le 1,\,i\in I\}$ with $0\in \intr(K)$ is a maximal $(b+\Z^n)$-free convex set, then $\dim(\rec(K))<n$, hence $\intr(\rec(K))=\emptyset$. Therefore, by Theorem \ref{th-representations}, the gauge of $K$ is the \emph{unique} representation of $K$.

However, even in the case $S=\Z^n_+$, uniqueness does not hold. Basu, Conforti, Cornu\'ejols and Zambelli~\cite{bccz2} show that in this case a maximal $S$-free convex set is always a polyhedron $K$, but in general $K$ admits an infinite set of representations; see also \cite{conforti2013cut}.

As an example, consider the polyhedron $K=\{r\in \R^2\st r_1\le 1,r_2\le 1,\frac{1}{2}(-r_1+r_2)\le 1\}$.  Then the function $\gamma_K$ defined as
$\gamma_K(r)=\max\{0,r_1,r_2,\frac{1}{2}(-r_1+r_2)\}$ is a representation of $K$ and indeed $\gamma_K$ is the gauge of $K$.
However, the function $\mu_K$ defined as $\mu_K(r)=\max\{r_1,r_2,\frac{1}{2}(-r_1+r_2)\}$ is also a representation of $K$, and we will see that $\mu_K$ is the {\em smallest representation} of $K$, in the sense that $\mu_K\le \rho_K$ for every representation $\rho_K$ of $K$.
\smallskip

The \emph{polar} of a set $V\subseteq \R^n$  is the set $V^{\circ}=\{r\in \R^n\st r d\le 1 \mbox{ for all $d\in V$}\}$. A  set $G\subseteq \R^n$ is  a \emph{prepolar} of $V$ if $G^{\circ}=V$. If $V$ is a closed convex set and $0\in \intr(V)$, then $V^{\circ}$ is a bounded set and  $(V^{\circ})^{\circ}=V$. Therefore in this case  the polar of $V$ is itself a prepolar, but $V$ may have other prepolars.

For the polyhedron $K$ given in the above example, $K^{\circ}=\conv\{0,r_1,r_2,\frac{1}{2}(-r_1+r_2)\}$ and  $(K^{\circ})^{\circ}=K$, so $K^{\circ}$ is a prepolar of $K$. However, let $G=\conv\{r_1,r_2,\frac{1}{2}(-r_1+r_2)\}$. Then $G^{\circ}=K$, so $G$ is also a prepolar of $K$.

The {\em support function} of a set $G\subset\R^n$ is
\begin{equation}\label{support}
\sigma_G(r):=\sup_{d\in G}d r\,.
\end{equation}
The support function is  sublinear and remains unchanged if $G$ is replaced by its
closed convex hull: $\sigma_G=\sigma_{\overline{\conv}{(G)}}$. Conversely, any
sublinear function $\sigma$ is the support function of a
closed convex set, defined~by
$$G_\sigma:=\bigl\{d\in\R^n\,:\:d r\leqslant\sigma(r)\mbox{ for all
  $r\in\R^n$}\bigr\}.$$

\begin{theorem}
Let  $V$ be a closed convex set with $0\in \intr(V)$. Then $V$ admits an inclusion-wise smallest prepolar. The smallest representation of $V$ is the support function of the smallest prepolar of $V$.
\end{theorem}

If $V$ is a polyhedron $K$ and  $0\in \intr(K)$,  the support function of the smallest prepolar of $V$ can be computed with formula \eqref{eq-gauge-poly}.

\begin{theorem}
Let $ a_ix\le 1, i\in I$ be an irredundant representation of   a polyhedron $K$ with $0\in \intr(K)$. Then $\{a_i,\,i\in I\}$ is the smallest prepolar of $K$. Hence the smallest representation of $K$ is the function $\mu_K$ defined as \begin{equation}\label{eq:psi-formula} \mu_K(r)=\max_{i\in I}a_ir.\end{equation}
\end{theorem}

For more details on the theory of smallest representations, see Basu, Cornu\'ejols and Zambelli \cite{basu2011sublinear} and Conforti et al.~\cite{conforti2013cut}.

\subsubsection{Maximal $S$-free convex sets}

Theorem~\ref{thm:DBS} extends to more general sets $S$. Specifically, when $S=(b+\Z^n)\cap Q$, where $Q$ is a rational polyhedron, the authors in \cite{bccz2} prove that every maximal $S$-free convex set is a polyhedron with at most $2^n$ facets. Mor\'an and Dey \cite{moran2011maximal} showed that the same statement holds when $S=(b+\Z^n)\cap C$, where $C$ is a convex set. We illustrate below how this result can be obtained by means of a connection with the Helly number studied by Averkov \cite{averkov2013maximal}.
\medskip

For a set $S\subseteq \R^n$, Averkov introduces the following definitions and parameters.

\begin{definition} [$f(S)$, largest number of facets in a $n$-dimensional maximal $S$-free convex set]  If every $n$-dimensional maximal $S$-free convex set is a polyhedron with at most $k$ facets, $f(S)$ is the minimal $k$ as above. If there
exist no $n$-dimensional maximal S-free convex sets (e.g., for $S=\R^n$), define $f(S)=-\infty$.  If
there exist maximal $S$-free convex sets which are not polyhedra or maximal $S$-free polyhedra
with an arbitrarily large number of facets, define $f(S)=+\infty$.
\end{definition}

A set $A\subseteq \R^n$ is called \emph{$S$-convex} if it is of the form $A=S\cap C$ for some convex set $C\subseteq\R^n$.

\begin{definition} [Helly number] Given a nonempty family $\mathcal{F}$
of sets, the Helly number $h(\mathcal{F})$ of $\mathcal{F}$ is defined as follows. If $\mathcal{F}=\emptyset$, $h(\mathcal{F})=0$. If $\mathcal{F}\ne \emptyset$ and there exists $k$ such that
\[F_1\cap\dots \cap F_m=\emptyset \Longrightarrow \exists i_1,\dots, i_{\ell} \in [m], \ell\le k, \mbox{ such that } F_{i_1}\cap\dots\cap F_{i_{\ell}}=\emptyset\]
for all $F_1\dots,F_m\in \mathcal{F}$, then $h(\mathcal{F})$ is the minimal $k$ as above. In all other cases, $h(\mathcal{F})=+\infty$.
\end{definition}

For $S\subseteq \R^n$ we use the notation
$ h(S) = h(\{S\cap C\colon C\subseteq \R^n\mbox{ is convex}\}) $.
That is, $h(S)$ is the Helly number of the family of all $S$-convex sets.
Note that if $C$ is a convex set, then $h(S\cap C)\le h(S)$. When $S=\Z^n$, Doignon~\cite{Doignon1973} proves $f(S) = h(S) = 2^n$ -- see Theorems~\ref{thm:DBS} and~\ref{th:h-Zn}.

Averkov \cite{averkov2013maximal} proves the following.

\begin{theorem}\label{TH-FlessH}
Given $S\subseteq \R^n$, $f(S)\le h(S)$.
\end{theorem}

\begin{proof} We prove the theorem under the assumption that every maximal $S$-free convex set is a polyhedron.
Let $K\subseteq \R^n$ be a maximal $S$-free polyhedron. Represent $K$ as the intersection of closed half-spaces $H_1, \dots, H_m$. Then $S\cap \intr(H_1)\cap \dots \cap\intr(H_m)=\emptyset$. By the definition of Helly number, there exist indices $i_1,\dots, i_{\ell} \in [m]$, with $\ell\le h(S)$, such that
$S\cap\intr(H_{i_1})\cap\dots \cap\intr(H_{i_{\ell}})=\emptyset$. It follows that $K\subseteq K':=H_{i_1}\cap\dots\cap H_{i_{\ell}}$, where $K'$ is an $S$-free polyhedron with at most $h(S)$ facets. By maximality of $K$, we have $K=K'$ and thus $K$ has at most $h(S)$ facets.

In the more general case in which maximal $S$-free convex sets are not guaranteed to be polyhedra, Averkov \cite{averkov2013maximal} derives inequality $f(S)\le h(S)$ by approximating a maximal $S$-free convex set $K$ with a sequence of polyhedra $K_t\subseteq K$ that converges to $K$.
\end{proof}

Averkov \cite{averkov2013maximal} proves that if $S$ is a discrete set (i.e., $S\cap B$ is finite for every bounded set $B\subseteq\R^n$), then $h(S)=f(S)$. We sketch his proof below.

\begin{theorem}\label{TH-h=f}
If $S\subseteq\R^n$ is discrete, then $h(S)=f(S)$.
\end{theorem}

\begin{proof}
For the sake of simplicity, we assume that $S$ is finite; Averkov \cite{averkov2013maximal} then shows how to derive the result for discrete sets $S$ by means of limit arguments.

Since $S$ is finite, every $S$-convex set can be written as the intersection of $S$ with finitely-many open half-spaces.
By using this, one easily verifies that $h(S)=h(\mathcal F)$, where $\mathcal F$ is the family of all sets of the form $S\cap H$, with $H$ being an open half-space. Thus it is enough to prove that $h(\mathcal F)\le f(S)$.

Let $H_1,\dots,H_m$ be open half-spaces such that $H_1\cap\dots\cap H_m\cap S=\emptyset$. We need to show that there exists $I\subseteq\{1,\dots,m\}$ with $|I|\le f(S)$ such that $\bigcap_{i\in I}H_i\cap S=\emptyset$.
For $i=1,\dots,m$, $H_i=\{x\in\R^n\colon a_i x<\beta_i\}$ for some $a_i\in\R^n$ and $\beta_i\in\R$.

For $i=1,\dots,n$ we do the following. Define $\gamma_i$ as the supremum of the values $\gamma $ such that, if $H_i$ is replaced with $\{x\in\R^n\colon a_ix<\gamma\}$, then we still have empty intersection with $S$. Let $I$ be the set of indices for which $\gamma_i<+\infty$. For $i\in I$, redefine $H_i:=\{x\in\R^n\colon a_ix<\gamma_i\}$. One verifies that $\bigcap_{i\in I}H_i\cap S=\emptyset$.

By definition of $\gamma_i$ and by the finiteness of $S$, one can check that every inequality $a_ix\le\gamma_i,i\in I$ defines a facet of the polyhedron $K=\{x\in\R^n\colon a_ix\le\gamma_i,\,i\in I\}$, and every such facet contains an integer point in its relative interior. Then $K$ is a maximal $S$-free convex set and thus has at most $f(S)$ facets. In other words, $|I|\le f(S)$, and thus $h(\mathcal F)\le f(S)$.
\end{proof}

\begin{theorem}\label{th:h-Zn}
$h(b+\Z^n)=2^n$ for every $b\in\R^n$.
\end{theorem}

\begin{proof}
It suffices to consider the case $b=0$.
By Theorem \ref{thm:DBS}, $f(\Z^n)\le2^n$. Since there exist maximal lattice-free polyhedra with $2^n$ facets, $f(\Z^n)=2^n$. Then Theorem \ref{TH-h=f} yields $h(\Z^n)=f(\Z^n)=2^n$.
\end{proof}


We can now prove the result by Dey and Mor\'an \cite{moran2011maximal}.

\begin{theorem}\label{th-poly-convex}
Let $S=(b+\Z^n)\cap C$, where $b\in\R^n$ and $C$ is a convex set. Then $f(S)\le2^n$.
\end{theorem}

\begin{proof}
Recall that $h(S)\le h(b+\Z^n)$, as $S=(b+\Z^n)\cap C$ and $C$ is a convex set. Then, by Theorems \ref{TH-h=f} and \ref{th:h-Zn}, $f(S)=h(S) \le h(b+\Z^n)=2^n$.
\end{proof}

Averkov also shows an extension of this result, namely $f(S)\le 2^n$ for every $(\Z^n\times\R^p)$-convex set $S$~\cite{averkov2013maximal}. More recently, Aliev, Bassett, De Loera, Louveaux generalize Theorem~\ref{thm:DBS} in the following way~\cite{aliev2016quantitative}. Given natural numbers $n, k$, they prove the existence of a constant $c(n,k)$ (depending only on $n,k$) such that any maximal polyhedron with exactly $k$ integer points in its interior has at most $c(n,k)$ facets.

We mention that for $S=b+\Z^2$, a complete classification of maximal $S$-free convex sets in $\R^2$ has been obtained in~\cite{dw2008} by Dey and Wolsey; Cornu\'ejols and Margot~\cite{cm} give an alternate proof. The classification states that any maximal $(b+\Z^2)$-free convex set is one of five different types:
\begin{enumerate}
\item A {\em split}, i.e., the intersection of two half spaces whose corresponding hyperplanes are parallel and contain infinitely many points from $S$.
\item A {\em type 1} triangle, which is an affine unimodular transformation of $\conv\{0, 2e^1, 2e^2\}$.
\item A {\em type 2} triangle, which has a single side with multiple points from $S$ in its relative interior, and the other two sides have exactly one point from $S$ in their relative interior. Moreover, the line passing through these two points is parallel to the third side.
\item A {\em type 3} triangle, which contains exactly three points from $S$ on its boundary, one in the relative interior of each side.
\item A {\em quadrilateral} where each side has exactly one point from $S$ in its relative interior, and the four points form the translation of a fundamental parallelepiped of $\Z^2$.
\end{enumerate}

 The classification is also completely known in $\R^2$ when $S = (b+\Z^2) \cap Q$, where $Q$ is a rational polyhedron~\cite{basu-paat-lifting}. A partial classification for $\R^3$ when $S = b+\Z^3$ has been obtained in~\cite{MR2855866}, which provides a complete description of all the $S$-free tetrahedra with integral vertices (which extend the type 1 triangles in $\R^2$). However, a complete classification for $\R^3$ has not been obtained.

  When $S=(b+\Z^n)\cap Q$, where  $Q$ is a rational polyhedron, Basu et al.~\cite{bccz2} prove the following sharpening of Theorem \ref{th-poly-convex}.
\begin{theorem}\label{thm:S-free} Let $S=(b+\Z^n)\cap Q$, where  $Q$ is a rational polyhedron such that $\dim(\conv(S))=n$. A set $K\subseteq \R^n$ is a maximal $S$-free convex set if and only if one of the following holds:
\begin{enumerate}[\upshape(i)]
\item $K$ is a polyhedron such that $K\cap \conv(S)$ has nonempty interior,  $K$ does not contain any point of $S$ in its interior and there is a point of $S$ in the relative interior of every facet of $K$. The  cone $\rec(K\cap\conv(S))$ is rational and it is contained in $\lin(K)$.
\item $K$ is a half-space of $\R^n$ such that $K\cap \conv(S)$ has empty interior and the boundary of $K$ is a supporting hyperplane of $\conv(S)$.
\item $K$ is a hyperplane of $\R^n$ such that $\lin(K)\cap\rec(\conv(S))$ is not a rational polyhedron.
\end{enumerate}
\end{theorem}

\section{The pure integer model}\label{sec:pure}

 In this section, we consider  the pure integer model \eqref{def integer set}, which we rewrite here for convenience:
\begin{equation}
I_S(P):= \setcond{y \in  \Z_+^\ell} { Py \in S}.
\end{equation}
We assume throughout this section that $S =b+\Z^n$ for some $b\in\R^n\setminus\Z^n$. This case was introduced and studied by Gomory and Johnson \cite{infinite,infinite2}. In the literature, this model is frequently refereed to under the name of \emph{infinite group problem}. Recently, Yildiz and Cornu\'ejols extend the analysis to more general $S$~\cite{yildiz2015integer}, but we will not cover their work in this survey.

Although for $S=b+\Z^n$ there do exist integer valid functions that take negative values,  \emph{we only concentrate on nonnegative integer valid functions in this survey.} Some justification for the nonnegativity assumption can be given as follows. If $P$ is a rational matrix and $I_S(P)\ne \emptyset$, then $\rec(\conv( I_S(P)))=\R^\ell_+$~\cite{conforti2014integer}. Therefore every  inequality  that is essential for  a linear description of $\conv( I_S(P))$ has nonnegative coefficients. 

\subsection {Minimal integer valid functions}

Recall that an integer valid function $\pi$ for $S$ is said to be \emph{minimal} if there is no integer valid function $\pi' \neq \pi$
such that $\pi'(p) \le \pi(p)$ for all $p \in \R^n$. Notice that if  $\pi$ is a nonnegative integer valid function which is minimal, then $\pi\le 1$.
Minimal integer valid
functions for $S$ were characterized by Gomory and
Johnson~\cite{infinite}.

Recall that a function $\pi\colon \R^n \rightarrow \mathbb{R}$ is \emph{subadditive} if
$\pi(p^1 + p^2) \le \pi(p^1) + \pi(p^2)$ for all $p^1, p^2 \in \R^n$. When $S = b+\Z^n$, we say that  $\pi$  satisfies the \emph{symmetry condition} if $\pi(p) + \pi(b - p) = 1$ for all $p \in \R^n$. Finally, $\pi$ is \emph{periodic modulo
  $\Z^n$} if $\pi(p) = \pi(p + w)$ for all $w \in \Z^n$.

\begin{theorem}[Gomory and Johnson \cite{infinite}] \label{thm:minimalinteger} Let $S = b+\Z^n$ for some $b\notin\Z^n$, and let
  $\pi \colon \R^n \rightarrow \mathbb{R}$ be a nonnegative function. Then $\pi$
  is a minimal integer valid function for $S$ if and only if $\pi(w) = 0$ for
  all $w\in \Z^n$, $\pi$ is subadditive, and $\pi$ satisfies the symmetry
  condition. (These conditions imply that $\pi$ is periodic modulo $\Z^n$
  and $\pi(b+w)=1$ for every $w\in \Z^n$.)
\end{theorem}

\begin{proof} We first prove the ``only if'' part of the statement. Assume that $\pi $ is a minimal integer valid function for $S$. We need to show the following three facts.\medskip

\noindent (a) {\em $\pi (w) = 0$ for every $w \in \Z^n$}. Define by $\pi'(p) = \pi(p)$ if $p\not\in \Z^n$ and $\pi'(w) = 0$ for all $w \in \Z^n$. If $\bar y \in \Z^\ell_+$ is a point in $I_S(P)$ for some $P$, then so is $\tilde y$ defined by $\tilde y_p = \bar y_p$ if $p \not\in \Z^n$, and  $\tilde y_p = 0$ if $p \in \Z^n$. Moreover, $\sum \pi'(p)\bar y_p = \sum \pi(p) \tilde y_p \geq 1$ since $\pi$ is valid. Therefore, $\pi'$ is valid. Minimality of $\pi$ implies $\pi'=\pi$. \medskip

\noindent (b) {\em $\pi$ is subadditive}. Let $p^1, p^2 \in \R^n$. We need to show $\pi(p^1)+\pi(p^2)\geq \pi(p^1+p^2)$. This inequality holds when $p^1=0$ or $p^2=0$ because $\pi (0) = 0$. Assume now that $p^1 \not= 0$, $p^2 \not= 0$ and $\pi(p^1)+\pi(p^2) < \pi(p^1+p^2)$. Define the function $\pi'$ as follows:
$\pi'( p^1+p^2 ) =\pi(p^1)+\pi(p^2)$ and
$\pi'(p)=\pi(p)$ for $p\ne  p^1+p^2 $.
The same argument used in Lemma \ref{Le:Cont-oldProperties} shows that $\pi'$ is valid.
\medskip

Now (a) and (b) imply that $\pi$ is periodic. This is because for any $p\in \R^n$ and $w \in \Z^n$, $\pi(p+w) \leq \pi(p) + \pi(w) = \pi(p)$ where the inequality is from (b) and the equality is from (a). Similarly, $\pi(p+w) = \pi(p+w) + \pi(-w) \geq \pi(p+w - w) = \pi(p)$.
\medskip

\noindent (c) {\em $\pi$ satisfies the symmetry condition.} Suppose there exists $\tilde p \in \R^n$ such that $\pi (\tilde p) + \pi (b-\tilde p) \not= 1$.
Since $\pi$ is valid, $\pi (\tilde p) + \pi (b- \tilde p) = 1 + \delta$ where $\delta > 0$.  Since $\pi(p)\leq 1$ for all $p\in\R^n$  for any minimal function $\pi$, it follows that $\pi(\tilde p)>0$. Define the function $\pi'$ by
$$\pi'(p) := \left\{ \begin{array}{ll}
\frac 1{1+\delta} \pi(\tilde p) &
\mbox{if } p = \tilde p, \\
\pi(p) & \mbox{if } p \not= \tilde p. \end{array} \right.$$
We show that $\pi'$ is valid. Consider any $\bar y  \in I_S(P)$ for some matrix $P$ containing column $\tilde p$. Note that
$$\sum \pi'(p)\bar y_p=\mathop{\sum}_{p\neq \tilde p} \pi(p)\bar y_p+\frac{1}{1+\delta}\pi(\tilde p)\bar y_{\tilde p}.$$

If $\bar y_{\tilde p}=0$ then $\sum\pi'(p)\bar y_p=\sum \pi(p)\bar y_p\geq 1$ because $\pi$ is valid. If $\bar y_{\tilde p}\geq (1+\delta)/\pi(\tilde p)$ then $\sum\pi'(p)\bar y_p\geq 1$.
Thus we can assume that $1\leq\bar y_{\tilde p} < (1+\delta)/\pi(\tilde p)$.

Observe that $\mathop{\sum}_{p\neq \tilde p} \pi(p)\bar y_p+\pi(\tilde p) (\bar y_{\tilde p}-1) \geq \mathop{\sum}_{p\neq \tilde p} \pi(p \bar y_p)+\pi(\tilde p (\bar y_{\tilde p}-1)) \geq  \pi(\mathop{\sum}_{p\neq \tilde p} p \bar y_p + \tilde p (\bar y_{\tilde p}-1)) =  \pi(b -\tilde p)$, where the inequalities
follow by the subadditivity of $\pi$ and the equality follows as $\pi$ is periodic modulo $\Z^n$ and $\mathop{\sum}_{p\neq \tilde p} p \bar y_p + \tilde p \bar y_{\tilde p} \in S$. Therefore
\begin{eqnarray*}
\sum\pi'(p)\bar y_p &=&\mathop{\sum}_{p\neq \tilde p} \pi(p)\bar y_p+\pi(\tilde p) (\bar y_{\tilde p}-1)+\pi(\tilde p)-\frac{\delta}{1+\delta}\pi(\tilde p)\bar y_{\tilde p}\\
&\geq & \pi(b-\tilde p)+\pi(\tilde p)-\delta \\
&= & 1+\delta -\delta=1.
\end{eqnarray*}
This shows that $\pi'$ is valid, contradicting the minimality of $\pi$.
\smallskip

We now prove the ``if'' part of the statement. Assume that $\pi (w) = 0$ for all $w \in \Z^n$, $\pi$ is subadditive, and satisfies the symmetry condition. As noted earlier, the first two conditions imply that $\pi$ is periodic.

We first show that $\pi$ is valid. The symmetry condition implies $\pi (0) + \pi (b) =1$. Since $\pi (0) = 0$, we have $\pi (b) =1$. Let $P$  and $\bar y$ be such that $P\bar y = b + w$ for some $w\in \Z^n$. We have that $\sum \pi(p) \bar y_p \geq \pi( \sum p \bar y_p ) = \pi(b+w)=\pi(b) = 1$, where the inequality comes from subadditivity and the second to last equality comes from periodicity. Thus $\pi$ is valid.

To show that $\pi$ is  minimal, suppose by contradiction that there exists an integer valid function $\pi' \leq \pi$ such that $\pi' (\tilde p) < \pi ( \tilde p)$ for some $\tilde p \in \R^n$. Then
$\pi (\tilde p) + \pi (b -\tilde p ) =1$ implies $\pi' (\tilde p) + \pi' (b -\tilde p ) < 1$, contradicting the validity of $\pi'$.
\end{proof}

The above proof follows the one given in~\cite{corner_survey}.

\subsection{Extreme functions: Techniques for proving extremality}\label{s:roadmap}

When describing a full-dimensional polyhedron $K$, one is only interested in identifying the facet-defining inequalities of $K$, as all other valid inequalities can be expressed as convex combinations of facet-defining inequalities. In other words, the facet-defining inequalities form extreme rays of the cone of valid inequalities for $K$. In our context, the analogous notion is that of an {\em extreme function}. An integer valid function~$\pi$ is \emph{extreme}
for $S$ if it cannot be written as a proper convex combination of two other
integer valid functions for $S$, i.e., if $\pi = \tfrac12(\pi^1 + \pi^2)$ for integer valid functions $\pi^1, \pi^2$ implies  $\pi = \pi^1 = \pi^2$. Extreme functions are easily seen to be minimal.

\begin{remark} For the continuous model, extreme functions are defined in the same way and  were characterized by Cornu\'ejols and  Margot \cite{cm} for the case $n=2$. See also~\cite[Theorem 1.5]{bccz} for a result in general dimension $n$, which states that a valid function for the continuous model is extreme if and only if a certain restriction of this function gives a facet defining inequality for a well-defined polyhedron. \end{remark}

The following lemma will be useful in analyzing extreme functions.

\begin{lemma}
\label{lem:tightness}\label{lem:minimality-of-pi1-pi2}\label{lemma:tight-implies-tight}\label{Theorem:functionContinuous} \label{lem:lipschitz}
Let $S = b+\Z^n$ for some $b\notin\Z^n$.  Let  $\pi$ be a nonnegative minimal integer valid function for $S$ and suppose $\pi^1$ and $\pi^2$ are nonnegative integer valid functions such that $\pi = \frac12(\pi^1+\pi^2)$.  Then the following hold:
  \begin{enumerate}[\upshape(i)]
  \item $\pi^1,\pi^2$ are minimal~\cite{infinite}.
  \item 

  Let the \emph{additivity domain} of~$\pi$ be:
\begin{equation}
\label{eq:Epi}
  E(\pi) := \setcond{(x, y)} { \pi(x) + \pi(y) =\pi(x + y)}.
\end{equation}
  Then $E(\pi) \subseteq E(\pi^1) \cap E(\pi^2)$~\cite{infinite}.
  \item Suppose there exists a real number $M$ such that  $\limsup_{h \to 0^+}  \frac{\pi(h \rx)}{h} \leq M$ for all $\rx \in \R^n$ such that $\lVert r \rVert = 1$. Then $\pi$ is Lipschitz-continuous.  Furthermore, this condition holds for $\pi^1$ and $\pi^2$, and $\pi^1, \pi^2$ are Lipschitz-continuous~{\cite[Theorem 2.9]{basu-hildebrand-koeppe:equivariant}}.
  \end{enumerate}
\end{lemma}

To prove that a function $\pi$ is extreme, the main idea is to establish that $E(\pi) \subseteq E(\pi')$ implies $\pi = \pi'$ for every minimal integer valid function $\pi'$. Then, starting from the assumption $\pi = \frac12(\pi^1 + \pi^2)$, by Lemma~\ref{lem:tightness} (ii) $E(\pi) \subseteq E(\pi^1)$ and $E(\pi) \subseteq E(\pi^2)$, and therefore $\pi = \pi^1 = \pi^2$.

The main tools for establishing that $E(\pi) \subseteq E(\pi')$ implies $\pi = \pi'$ for every minimal integer valid function $\pi'$ are results that are collectively called {\em interval lemmas} (after the one-dimensional interval lemma of Gomory and Johnson~\cite{tspace}), and are presented in the next subsection.

We illustrate this framework for proving extremality by outlining a proof of the {\em $(n+1)$-slope theorem} of Basu, Hildebrand, K\"oppe and Molinaro~\cite{basu-hildebrand-koeppe-molinaro:k+1-slope}, which is one of the most general sufficient conditions for extremality of minimal integer valid functions.

\subsubsection{Regular solutions to Cauchy's functional equation}\label{s:real-analysis}

As mentioned above, the key to establishing extremality of a function $\pi$ is to prove that $E(\pi) \subseteq E(\pi')$ implies $\pi = \pi'$ for every minimal integer valid function $\pi'$. The first step is to show that $E(\pi) \subseteq E(\pi')$ implies that $\pi'$ has an affine linear structure whenever $\pi$ has such structure. For this purpose, we consider full-dimensional convex subsets $F \subseteq E(\pi') \subseteq \R^n \times \R^n$; therefore, $\pi(u) + \pi(v) = \pi(u+v)$ for all $(u,v) \in F$. This leads to the study of functions $\theta\colon\R^n \to \R$ satisfying:
\begin{equation}\label{eq:cauchy}
  \theta(u)+\theta(v) = \theta(u+v),\quad (u,v) \in F
\end{equation}
for a given subset $F\subseteq \R^n\times \R^n$. This equation is known as the \emph{(additive) Cauchy functional equation}.
\smallskip

The Cauchy functional equation is classically studied for
functions $\theta\colon\R\to\R$, where the additivity domain~$F$ is the entire
space~$\R\times\R$ (see, e.g., \cite{aczel1989functional}).
In addition to the \emph{regular solutions}, which are
the (homogeneous) linear functions $\theta(x) = cx$ for any $c\in\R$, there exist certain
pathological solutions, which are highly discontinuous~\cite[Chapter 2, Lemma 3]{aczel1989functional}. In order to rule out
these solutions, one imposes a regularity hypothesis.  Various such regularity
hypotheses have been proposed in the literature; for example, it is sufficient
to assume that the function~$\theta$ is bounded on every bounded interval
\cite[Chapter 2, Theorem 8]{aczel1989functional}. 
\smallskip

We now return to functions $\theta\colon \R^n\to \R$  and recall the notion of affine functions over a domain.
\begin{definition}
Let $U \subseteq \R^n$. We say that $\theta\colon U \to \R$ is {\em affine} over $U$ (with gradient $c$) if there exists $c \in \R^n$ such that
for any $u_1, u_2 \in U$ we have $$\theta(u_2) - \theta(u_1) =  c (u_2 - u_1).$$
\end{definition}
Equivalently, there exists $b\in \R$ such that  $\theta(u)=cu+b$ for every $u\in U$.

We define three projection operators on $\R^n \times \R^n$. For any subset $F\subseteq \R^n \times \R^n$, define $p_1(F) = \{ u \in \R^n \colon (u,v) \in F\}$, $p_2(F) = \{ v \in \R^n \colon (u,v) \in F\}$, and $p_3(F) = \{ u + v \colon (u,v) \in F\}$.


\begin{lemma}[Convex additivity domain lemma]{~\cite[Theorem 2.11]{bhk-IPCOext}}\label{lem:projection_interval_lemma_fulldim}
  Let $\theta \colon \R^n \to \R$ be a bounded function.
  Let $F \subseteq \R^n \times \R^n$ be a full-dimensional convex set
  such that $\theta(u) + \theta(v) = \theta(u+v)$ for all $(u, v) \in F$.
  Then there exists a vector $c\in \R^n$ such that $\theta$ is
  affine with the same gradient $c$ over $\intr(p_1(F))$,
  $\intr(p_2(F))$ and $\intr(p_3(F))$, respectively.
\end{lemma}

The special case of the above result in which $F = U\times V$ is the cartesian product of two closed proper intervals $U,V \subseteq \R$ is the so-called {\em interval lemma} stated in~\cite{tspace} (see also \cite[Lemma 6.26]{conforti2014integer}). Early extensions to this original result were made in Dey and Richard's work~\cite[Proposition 23]{dey3} and Dey et al's work~\cite[Proposition 3, Corollary 1]{dey1}.

\subsection{Sufficient conditions for extremality: the $(n+1)$-slope theorem}\label{s:sufficient-cond}

One of the most celebrated results in the study of extreme functions is the
so-called {\em Gomory-Johnson 2-slope theorem}~\cite{infinite} (see also \cite[Theorem 6.27]{conforti2014integer}), which states
that for $n=1$, if a continuous piecewise linear minimal valid integer function has only 2 values for
the derivative wherever it exists (2 slopes), then the function is
extreme. This was generalized to $n=2$ by Cornu\'ejols and
Molinaro~\cite{3slope}, and to general $n$ by Basu, Hildebrand, K\"oppe and
Molinaro~\cite{basu-hildebrand-koeppe-molinaro:k+1-slope}. We present the general $(n+1)$-slope theorem here, along with the main ingredients of its proof.

We introduce the definition of \emph{polyhedral complex}, a classical notion from polyhedral geometry~\cite[Chapter 5]{ziegler}, to formalize the notion of piecewise linear functions over $\R^n$ for $n \geq 2$. 

\begin{definition}
\label{def:polyhedralComplex}
A  polyhedral complex in $\R^n$ is a collection $\mathcal P$ of polyhedra in $\R^n$ such that:
\begin{enumerate}[\upshape(i)]
\item if $I \in \mathcal P$, then all faces of $I$ are in $\mathcal P$,
\item the intersection $I \cap J$ of two polyhedra $I,J \in \mathcal P$ is a face of both $I$ and $J$,
\item any compact subset of $\R^n$ intersects only finitely many polyhedra in $\mathcal P$. 
\end{enumerate}
\end{definition}
A polyhedron $I$ from $\mathcal P$ is called a {\em face} of the complex.
 We call the maximal faces of $\mathcal P$ the {\em cells}
of~$\mathcal P$. A function $\pi\colon \R^n\to \R$ is {\em  continuous piecewise linear} if there exists a  polyhedral complex $\mathcal P$ such that $\cup_{I \in \mathcal P}I=\R^n$ and $\pi$ is an  affine  function over each of the cells of~$\mathcal P$ (thus automatically imposing continuity for the function).

We next define  genuinely $n$-dimensional functions on $\R^n$  and then indicate that for the analysis of minimal and extreme functions, it suffices to study genuinely $n$-dimensional functions. This notion was first introduced in~\cite{basu-hildebrand-koeppe-molinaro:k+1-slope}.


\begin{definition}\label{def:genk}
A function $\theta\colon \R^n \rightarrow \R$ is {\em genuinely $n$-dimensional}
if there does not exist a function $\varphi \colon \R^{n-1} \rightarrow \R$ and a linear
map $f \colon \R^n \rightarrow \R^{n-1}$ such that $\theta = \varphi\circ f$.
\end{definition}

\begin{remark}[Dimension reduction;  {\cite[Proposition B.9, Remark B.10]{bhk-IPCOext}}]
\label{remark:dimension-reduction}

The extremality/minimality question for $\pi$ that is not genuinely $n$-dimensional can be reduced to the same question for a lower-dimensional genuinely $\ell$-dimensional function (so $\ell < n$.) When $\mathcal P$ is a rational polyhedral complex, this reduction can be done algorithmically. Thus,  we can assume without loss of generality that a function is genuinely $n$-dimensional.
\end{remark}

We can now state the $(n+1)$-slope theorem.

\begin{theorem}[{\cite[Theorem 1.7]{basu-hildebrand-koeppe-molinaro:k+1-slope}}]\label{thm:k+1slope}
Let $\pi\colon \R^n \to \R$ be a  nonnegative minimal integer valid function that is continuous piecewise linear and genuinely $n$-dimensional with at most $n+1$ slopes, i.e., at most $n+1$ different values for the gradient of $\pi$ where it exists. Then $\pi$ is extreme and has exactly $n+1$ slopes.
\end{theorem}

 We outline the proof of Theorem \ref{thm:k+1slope}. Let $\pi$ be a continuous piecewise linear minimal integer valid function that is genuinely $n$-dimensional with at most $n+1$ slopes. Let $\mathcal P$ be the polyhedral complex associated with $\pi$.

\begin{enumerate}
\item Subadditivity and the  genuine $n$-dimensionality of $\pi$ imply that $\pi$ has exactly $n+1$ gradient values $\gp^1, \ldots, \gp^{n+1} \in \R^n$. This is a relatively easy step, see Lemma 2.11 in~\cite{basu-hildebrand-koeppe-molinaro:k+1-slope}.
\item (Compatibility step) Let $\pi^1, \pi^2$ be valid functions such that $\pi = \frac12(\pi^1 + \pi^2)$.
 For each $i = 1, \ldots, n+1$, define $\mathcal P_i\subseteq \mathcal P$ to be the polyhedral complex formed by all the cells (and their faces) of $\mathcal P$ where the gradient of $\pi$ is $\gp^i$. Show that there exist $\gt^1, \ldots, \gt^{n+1}$ such that $\pi^1$ is affine over every cell in $\mathcal P_i$ with gradient $\gt^i$.
\item (Gradient matching step) We then prove properties  of genuinely $n$-dimensional functions with $n+1$ slopes that lead to a system of  equations that are satisfied by the coefficients of $\gp^1, \ldots, \gp^{n+1}$ and $\gt^1, \ldots, \gt^{n+1}$. Then, it is established that this system of equations has a unique solution, and thus, $\gp^i = \gt^i$ for every $i=1, \ldots, n+1$.
\item For every $r \in \R^n$ there exist $\mu_1, \mu_2, \ldots, \mu_{n+1}$ such that $\mu_i$ is the fraction of the segment $[\0,r]$ that lies in $\mathcal P_i$. Thus, $$\pi(r) = \pi(\0) + \sum_{i=1}^{n + 1} \mu_i (\gp^i r) = \pi^1(\0) + \sum_{i=1}^{n+1} \mu_i (\gt^i r) = \pi^1(r).$$ This proves that $\pi = \pi^1$ and thus, $\pi = \pi^1 = \pi^2$, concluding the proof of Theorem~\ref{thm:k+1slope}.
\end{enumerate}

%
%
%
%

We now elaborate on Steps 2. and 3.

\paragraph{Compatibility step.} The analysis of step 1 also shows that for every $i=1, \ldots, n+1$, there exist $C_i \in \mathcal P_i$ such that $\0 \in C_i$. This means that for every gradient value, there is a cell containing the origin with that gradient. Fix an arbitrary $i \in \{1, \ldots, n+1\}$ and consider any cell $I \in \mathcal P_i$. Let $F = \{(u,v) \in \R^n \times \R^n\colon u \in C_i, v \in I, u+v \in I\}$. Then $F \subseteq E(\pi)$ since for $(u,v)$ such that $u \in C_i, v \in I, u+v \in I$, $\pi(\u) + \pi(v) - \pi(\u + v) = (\gp^i \u) + (\gp^i v + \delta) - (\gp^i(\u+v) + \delta) = 0$ for some $\delta\in\R$; here, we use the fact that $\pi$ is affine over $C_i$ and $I$ with gradient $g^i$, and the facts that $0 \in C_i$ and $\pi(0) = 0$. By Lemma~\ref{lem:tightness} (ii), $F \subseteq E(\pi^1)$; by Lemma~\ref{lem:tightness} (iii), $\pi^1$ is continuous (because $\pi$ satisfies the hypothesis of Lemma~\ref{lem:tightness} (iii) as $\pi$ is continuous piecewise linear). By Lemma~\ref{lem:projection_interval_lemma_fulldim}, applied to $F$ and $\theta = \pi^1$, and continuity of $\pi^1$, we obtain that $\pi^1$ is affine on $C_i = p_1(F)$ and $I = p_2(F) = p_3(F)$ with the same gradient. Since the choice of $I$ was arbitrary, this establishes that for every cell $I\in\mathcal P_i$, $\pi^1$ is affine with the same gradient; this is precisely the desired $\gt^i$.

\paragraph{Gradient matching step.} The system of equations of step 3 has two sets of constraints, the first of which follows from the condition that $\pi(b + \w) = 1$ for every $\w \in \Z^n$ (see Theorem \ref{thm:minimalinteger}).  The second set
of constraints is more involved. Consider two adjacent cells $I, I' \in
\mathcal{P}$ that contain a segment $[\x, \y] \subseteq \mathbb{R}^n$ in their
intersection. Along the line segment $[\x,\y]$, the gradients of $I$ and $I'$
projected onto the line spanned by the vector $\y - \x$ must agree; the second
set of constraints captures this observation. We will identify a set of vectors $r^1, \ldots, r^{n+1}$ such that every subset of $n$ vectors is linearly independent and such that each vector $r^i$ is contained in $n$ cells of $\mathcal P$ with different gradients. We then use the segment $[\0,r^i]$ to obtain linear equations involving the gradients of $\pi$ and $\pi'$. The fact that every subset of $n$ vectors is linearly independent will be crucial in ensuring the uniqueness of the solution to the system of equations.

      \begin{lemma}[{\cite[Lemma 3.10]{basu-hildebrand-koeppe-molinaro:k+1-slope}}]\label{lemma:kkmDirections}
                There exist vectors $r^1, r^2, \ldots, r^{n + 1} \in \R^n$ with the following properties:
                \begin{enumerate}[\rm(i)]
                        \item For every $i,j,\ell \in \{1, \ldots, n + 1\}$ with $j, \ell$ different from $i$, the equations $r^i \gp^{j} = r^i \gp^{\ell}$ and $r^i \gt^{j} = r^i \gt^{\ell}$ hold.

                        \item Every $n$-subset of $\{r^1, \ldots, r^{n+1}\}$ is linearly independent. 
                \end{enumerate}
        \end{lemma}
%

       The proof of Lemma~\ref{lemma:kkmDirections} uses a result known as the {\em Knaster-Kuratowski-Mazurkiewicz Lemma (KKM Lemma)} from fixed point theory, which exposes a nice structure in the gradient pattern of $\pi$.

\begin{lemma}[KKM \cite{kkm,fpt}]\label{lem:KKM}
        Consider a simplex $\conv(u^j)_{j = 1}^{d}$. Let $F_1, F_2, \ldots, F_d$ be closed sets such that for all $J \subseteq \{1, \ldots, d\}$, the face $\conv(u^j)_{j \in J}$ is contained in $\bigcup_{j \in J} F_j$. Then the intersection $\bigcap_{j = 1}^d F_j$ is non-empty.

\end{lemma}
%

       This lemma is applied to the facets of a certain simplex $\Delta$ containing the origin, and the closed sets $F_i = \bigcup_{I \in \mathcal{P}_i}(\Delta \cap I)$, for $i=1,\ldots, n+1$. For each facet of $\Delta$ indexed by $i = 1, \ldots, n+1$, the KKM lemma (with $d=n$) ensures the existence of a point $r^i \in \bigcap_{j \neq i} F_j$ on the facet indexed by $i$. These points give the vectors $r^1, \ldots, r^{n+1}$ from Lemma~\ref{lemma:kkmDirections}. The bulk of the technicality lies in proving that the chosen simplex and the sets $F_i$ satisfy the hypothesis of the KKM lemma. 


        We finally present the system of linear equations that we consider.

        \begin{cor}[{\cite[Corollary 3.13]{basu-hildebrand-koeppe-molinaro:k+1-slope}}] \label{lemma:linear-system}
                Consider any $n+1$ affinely independent vectors $\a^1, \a^2, \ldots, \a^{n + 1} \in b+\mathbb{Z}^n$. Also, let $r^1, r^2, \ldots, r^{n + 1}$ be the vectors given by Lemma \ref{lemma:kkmDirections}. Then there exist $\mu_{ij} \in \R_+$ for $i, j \in \{1, \ldots, n+1\}$, with $\sum_{j = 1}^{n + 1} \mu_{ij} = 1$ for all $i \in \{1, \ldots, n+1\}$, such that both $\gt^1, \ldots, \gt^{n+1}$ and $\gp^1, \ldots, \gp^{n+1}$ are solutions to the linear system
        \begin{equation}\label{eq:linear-system}
\begin{aligned}
\textstyle\sum_{j=1}^{n+1} (\mu_{ij}\a^i) \gs^j & = 1 &\qquad& \text{for all }  i \in \{1, \ldots, n+1\}, \\
r^i \gs^{j} - r^i \gs^{\ell} & = 0 & \qquad& \text{for all } i,j, \ell \in \{1, \ldots, n+1\} \textrm{ such that } i \neq j, \ell,
\end{aligned}
\end{equation}
with variables  $\g^1, \ldots, \g^{n+1}\in \R^n$.
        \end{cor}

        We remark that we can always find vectors $\a^1, \a^2, \ldots, \a^{n + 1} \in b+\mathbb{Z}^n$ such that the set $\a^1, \ldots, \a^{n+1}$ is affinely independent, so the system above indeed exists. Property (ii) in Lemma~\ref{lemma:kkmDirections} and the fact that $\a^1, \ldots, \a^{n+1}$ are affinely independent can be leveraged to show that \eqref{eq:linear-system} has either no solutions or a unique solution. However, the linear algebra is involved and we refer the reader to~\cite[Section~3.3]{basu-hildebrand-koeppe-molinaro:k+1-slope}. Since $\gp^1, \ldots, \gp^{n+1}$ is a solution, the conclusion is that the system has a unique solution.

\section{The mixed integer model}\label{sec:mixed}

We consider here the mixed integer model \eqref{def mixed-int set}:
\begin{equation*}
	X_{S}(R,P) := \setcond{(s,y) \in \R_+^k \times \Z_+^\ell }{ Rs + Py\in S }
\end{equation*}
where $S\subseteq \R^n\setminus \{0\}$ is a closed set, and $k$ and $\ell$ are both positive. Recall that a pair of functions $(\psi, \pi)$ is a  valid pair if and only if $\sum\psi(r)s_r + \sum\pi(p)y_p \ge 1$ is a valid inequality for $X_{S}(R,P)$, for every $k$, $\ell$, $R$ and $P$. We do not make any nonnegativity assumptions on the functions $\psi, \pi$ in this section, as was done in Section~\ref{sec:pure}. Note that if $(\psi, \pi)$  is a valid pair, then $\psi$ is a valid function and $\pi$ is an integer valid function. However, the converse does not hold.
\smallskip

When $S=b+\Z^n$,   Theorem \ref{thm:psi-B}  characterizes minimal valid functions $\psi$ for the continuous model, and Theorem \ref{thm:minimalinteger}  characterizes {\em nonnegative} minimal integer valid functions $\pi$ for the pure integer model. For the mixed integer model \eqref{def mixed-int set}, Johnson \cite{johnson} gives such a characterization for {\em nonnegative} minimal valid pairs.
\begin{theorem}\label{thm:psi,pi-min}  Let $S=b+\Z^n$ for some $b\notin\Z^n$, and let $(\psi,\pi)$ be a valid pair, with $\pi\ge0$. Then $(\psi,\pi)$ is a minimal  valid pair if and only if $\pi$ is a minimal integer valid function  and $\psi$ satisfies
\begin{equation*}\label{eq:psi-lim}
\psi(r)=\limsup_{h\rightarrow 0^+} \frac{\pi(h r)}{h} \quad\mbox{ for every } r\in \R^n.
\end{equation*}
\end{theorem}
\smallskip

When $S=b+\Z^n$, for the continuous model, Theorem \ref{thm:psi-B} shows that minimal functions $\psi$ are gauge functions of $K$, where polyhedron $K$ is a maximal $S$-free convex set, and Theorem \ref{TH-gaugePolyhedron} gives a formula for the computation of $\psi$.
One of the most studied procedures to compute a minimal pair $(\psi, \pi)$ for the mixed integer model  is to start from a minimal function $\psi$ for the continuous model, and then compute $\pi$ such that $(\psi, \pi)$ is a minimal pair. Such a procedure goes under that name of lifting. A  \emph{lifting} of a valid function $\psi$ is a function $\pi$ such that $(\psi,\pi)$ is a valid pair.
A lifting $\pi$ of $\psi$ is \emph{minimal} if  $\pi=\pi'$ for every lifting $\pi'$ of $\psi$ such that   $\pi'\le \pi$.

Theorem \ref{thm:psi,pi-min} shows that  if $(\psi, \pi)$ is a minimal valid pair  and $\pi\ge0$, then $\pi$ is a minimal integer valid function. However, $\psi$ is not guaranteed to be a minimal valid function. Therefore the lifting procedure outlined above, applied to a minimal valid function, only produces a subset of minimal valid pairs, which however includes the most well-known and computationally effective pairs.
\smallskip

 The idea of using lifting in this context was proposed by Dey and Wolsey~\cite{dw2008}, who imported the concept of {\em monoidal strengthening}. Monoidal strengthening was  introduced by Balas and Jeroslow~\cite{baljer} to strengthen cutting planes by using integrality information. Searching for minimal liftings $\pi$ is analogous to the idea of strengthening the ``trivial" valid pair $(\psi, \psi)$ by using the integrality information on the $y$ variables.


\subsection{A geometric view of lifting}
 The following proposition imposes regularity on the structure of minimal liftings.

 \begin{prop}[{\cite[Proposition A.3]{basu-paat-lifting}}]\label{prop:periodic}
 Given a closed set $S \subseteq \R^n\setminus\{0\}$, let $\psi$ be a valid function and $\pi$ a minimal lifting of $\psi$. Define
  \begin{equation}\label{eq:W}W_S := \{w\in \R^n \colon s+\lambda w\in S~, \forall s\in S, \forall \lambda\in \Z\}.\end{equation}
  Then $\pi(p + w) = \pi(p)$ for all $p \in \R^n$ and $w \in W_S$ (i.e., $\pi$ is periodic modulo $W_S$).
  \end{prop}
\begin{proof} Let $\pi$ be a minimal lifting of $\psi$. Assume to the contrary that there exists some $\hat{p}\in \R^n$ and $w\in W_S$ such that $\pi(\hat{p})\neq \pi(\hat{p}+w)$. Since $-w\in W_S$, we may assume $\pi(\hat{p})>\pi(\hat{p}+w)$. Define a function $\tilde{\pi}$ as $\tilde{\pi}(\hat p)=\pi(\hat{p}+w)$ and $\tilde{\pi}(p)=\pi(p)$ for $p\ne\hat{p}$. We show that $\tilde\pi$ is a lifting of $\psi$, a contradiction to the assumption that $\pi$ is a minimal lifting.

Let $ R\in \R^{n\times k}$, $P\in \R^{n\times \ell}$, and $(s,y)\in X_S(R,P)$. We assume, without loss of generality, that $\hat{p}$ and $\hat{p}+w$ are columns of $P$  and we  show that $\sum\psi(r)s_r + \sum\tilde{\pi}(p)y_p \ge 1$.
Let  $\bar y$ be such that $\bar{y}_{\hat{p}}=0$ and $\bar{y}_{p}=y_{p}$ for $p\ne \hat{p}$. Since $w\in W_S$ and $y_{\hat{p}} \in \Z$, it follows that
\begin{equation*}
Rs+Py \in S \iff Rs+Py+wy_{\hat{p}} = Rs+P\bar{y}+(\hat{p}+w)y_{\hat{p}} \in S.
\end{equation*}
Let $y'$ be defined as $y'_{\hat{p}+w}=y_{\hat{p}}+y_{\hat{p}+w}$, $y'_{\hat{p}}=0$ and $y'_p=y_p$ for $p\ne \hat{p}+w,\,\hat{p}$. Since $\pi$ is a lifting of $\psi$ and $Rs+Py'=Rs+P\bar{y}+(\hat{p}+w)y_{\hat{p}}\in S$, we have
\begin{equation*}
\sum\psi(r)s_r + \sum\tilde{\pi}(p)y_p = \sum\psi(r)s_r + \sum\pi(p)y'_p \ge 1.
\end{equation*}
\end{proof}

The set $W_S$ for arbitrary $S$ was defined first in~\cite{basu-paat-lifting}, where the above proposition is also proved.

Given a valid function $\psi$, the {\em lifting region} $ T_\psi$ (first introduced in~\cite{dw2008}) is:
\begin{equation}\label{eq:lifting-region} T_\psi := \{r \in \R^n \colon \pi(r) = \psi(r) \textrm{ for every minimal lifting }\pi \textrm{ of }\psi\}.\end{equation}

Conforti et al. \cite{ccz} prove that:
\begin{theorem}
If $S=(b+\Z^n)\cap Q$, where $b\notin\Z^n$ and $Q$ is a rational polyhedron, and $\psi$ is a minimal valid function, then there exists a full-dimensional ball $B(0,\varepsilon)\subseteq T_{\psi}$ for some $\epsilon>0$.
\end{theorem}

The next theorem provides a lifting  for any valid function $\psi$  and characterizes a condition on $T_\psi$ for this lifting to be minimal.
   \begin{theorem}[{\cite[Proposition A.3]{basu-paat-lifting}}]
 Given a closed set $S \subseteq \R^n\setminus\{0\}$, let $\psi$ be a valid function and define
  \begin{equation}\label{eq:pi-ast}
\pi^\ast(p) = \inf_{w \in W_S} \psi(p + w),\;\textrm{ for all } p\in \R^n.
\end{equation}
Then $\pi^\ast$ is a lifting of $\psi$. Moreover, if $T_\psi + W_S = \R^n$, then $\pi^\ast$ is a minimal lifting of $\psi$ and $\pi^\ast(p)=\psi(p+w)\mbox{  for any $w\in W_S$ such that $p+w\in T_{\psi}$}.$

\end{theorem}
  \begin{proof}
We first show that $(\psi, \pi^\ast)$ is a valid pair. Suppose to the contrary that there exist matrices $R, P$ and $(\bar s, \bar y) \in X_S(R,P)$ such that $\sum \psi(r)\bar s_r + \sum\pi^\ast(p)\bar y_p < 1$. Let $\epsilon = 1- \sum \psi(r)\bar s_r - \sum\pi^\ast(p)\bar y_p$. For each column $p$ of the matrix $P$, by definition of $\pi^\ast$, there exists $w_p \in W_S$ such that $\psi(p + w_p) \leq \pi^\ast(p) +\frac{\epsilon}{2\ell(\bar y_p + 1)}$ (recall that $\ell$ is the number of columns of $P$). Let $P'$ be the matrix with columns $p + w_p$. Then $R\bar s + P'\bar y = R\bar s + P\bar y + W\bar y$, where $W$ is the matrix with columns $w_p$. Since $W\bar y \in W_S$, $(\bar s, \bar y) \in X_S(R,P')$. However,
$$\begin{array}{rcl}1 \leq \sum \psi(r) \bar s_r + \sum \psi(p+w_p)\bar y_p & \leq & \sum \psi(r) \bar s_r + \sum (\pi^\ast(p) + \frac{\epsilon}{2\ell(\bar y_p + 1)})\bar y_p \\ & \leq &\sum \psi(r)\bar s_r + \sum\pi^\ast(p)\bar y_p + \frac\epsilon 2 \\ & = &1 - \epsilon + \frac\epsilon 2 < 1,\end{array}$$
where we used the fact that $(\psi,\psi)$ is a valid pair. Thus we have a contradiction, and therefore $\pi^\ast$ is a lifting of $\psi$.

Let $\pi$ be any minimal lifting. Consider any $p \in \R^n$ and let $w \in W_S$ be such that $p + w \in T_\psi$. By Proposition~\ref{prop:periodic}, $\pi(p) = \pi(p+w) = \psi(p+w) \geq \pi^*(p)$. This implies that $\pi(p) = \pi^*(p) = \psi(p+w)$ since $\pi$ is a minimal lifting.
\end{proof}

\paragraph{Importance of the lifting region.} In light of the above results, if we start with a valid function $\psi$ whose values we can compute, and an explicit description for $T_\psi$ can be obtained, then the coefficients $\pi^\ast(p)$ can be computed by finding  a $w\in W_S$ such that $p + w \in T_\psi$, and then using the formula for $\psi(p+w)$.
Moreover, if $\psi$ is a {\em minimal} valid function, then $(\psi, \pi)$ is a minimal valid pair. Basu et al.~\cite{bcccz}  show that when $S=(b+\Z^n)\cap Q$ where $Q$ is a rational polyhedron, $T_\psi$ can be described as the finite union of full dimensional polyhedra, each of which has an explicit inequality description. This is discussed below.

\subsubsection{A description of the lifting region}\label{s:desc}
 Throughout this section we assume $S=(b+\Z^n)\cap Q$ where $Q$ is a rational polyhedron, and we discuss the lifting region $T_{\psi}$ of a minimal  valid function $\psi$. Recall from Theorem \ref{thm:S-free}  that in this case maximal $S$-free convex sets are polyhedra.  Let $ a_i r \leq 1, i\in I$ be an irredundant description  of a maximal $S$-free polyhedron $K$ with $0\in \intr(K)$. Recall that the minimal valid function $\psi$ associated with $K$ is $\psi(r)=\max_{i\in I}a_ir$.
 \smallskip

For each $s \in K \cap S$, let $k(s) \in I$ be an index such that $a_{k(s)} s = 1$. By Theorem \ref{thm:S-free} such an index exists since $K$ is a maximal $S$-free convex set, and $s$ is on the boundary of $K$. Define the {\em spindle} $T(s)$ as follows:
$$T(s) := \{r\in \R^n \colon (a_i - a_{k(s)}) r \leq 0,\; (a_i - a_{k(s)}) (s - r) \leq 0 \; \mbox{ for all } i \in I\}.$$
Basu et al.~\cite{bcccz} prove the following:
\begin{theorem}  Let $S=(b+\Z^n)\cap Q$, where $b\notin\Z^n$ and $Q$ is a rational polyhedron, and let $K$ and $\psi$ be as above.  Then the lifting region is: $$T_{\psi}=T(S,K) := \bigcup_{s \in K\cap S} T(s).$$
  \end{theorem}

Figure~\ref{fig:regions} illustrates the region $T_\psi$ for several examples. We collect some basic properties of the lifting region that were presented in~\cite{bcccz}. Define $L_K = \{r \in \R^n\colon a_i r = a_j r \mbox{ for all } i,j \in I\}.$

\begin{figure}
\centering \subfigure[A maximal $(b+\Z^2)$-free triangle with three integer points] {\label{fig:region1}
\includegraphics[height=2.5in]{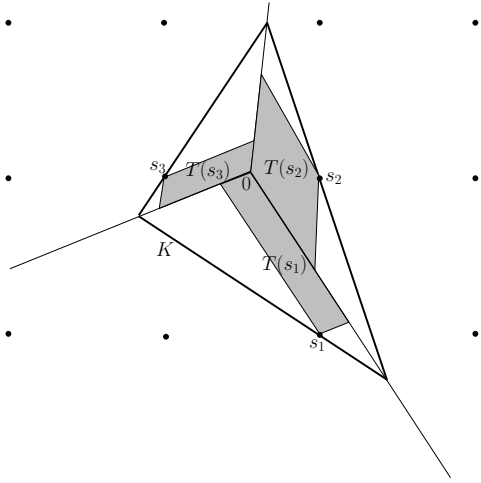}}
\hspace{0.5in}\subfigure[A maximal $(b+\Z^2)$-free triangle with integer
vertices] {\label{fig:region4}
\includegraphics[height=2.0in]{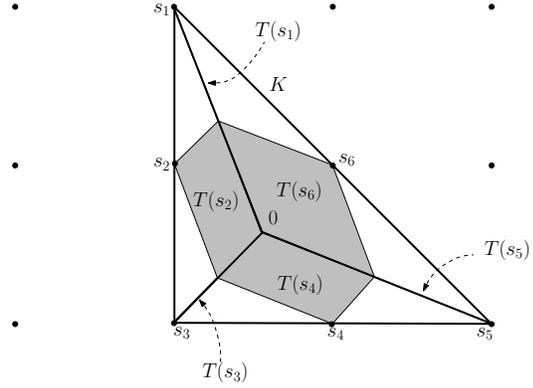}}
\hspace{0.5in}\subfigure[A wedge] {\label{fig:region2}
\includegraphics[height=2.5in]{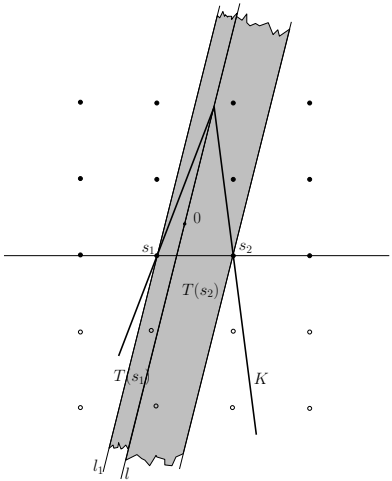}}
\hspace{0.5in} \subfigure[A truncated wedge] {\label{fig:region3}
\includegraphics[height=3.0in]{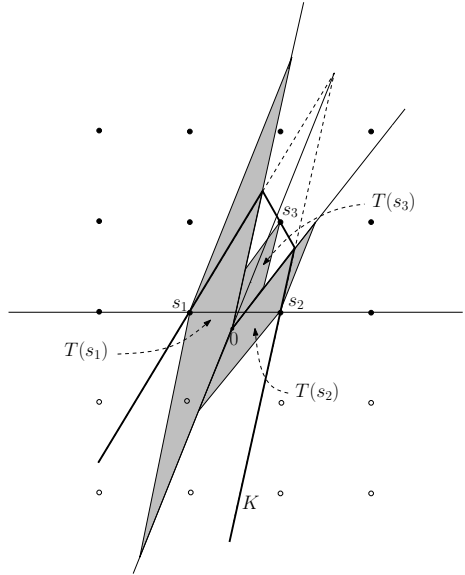}} \caption{Regions $T(s)$ for some
maximal $S$-free polyhedra $K$ in the plane and $s \in S \cap K$. In Figures~\ref{fig:region1} and~\ref{fig:region4}, $S = b+\Z^2$ for some $b \not\in \Z^2$ and in Figures~\ref{fig:region2} and~\ref{fig:region3} $S = (b+\Z^2) \cap H$ where $H$ is a half-space shown in the figures (dark circles show points from $S$ and hollow circles show points in $b+\Z^2$ that are not in $S$). The thick
dark line indicates the boundary of $K$. For a particular $s$, the dark gray regions denote $T(s)$. The jagged lines in a region indicate that it extends to infinity. For example,
in Figure~\ref{fig:region2}, $T(s_1)$ is the strip
between lines $l_1$ and $l$. Figure~\ref{fig:region4} shows an
example where $T(s)$ is full-dimensional for $s_2, s_4, s_6$, but is
not full-dimensional for $s_1, s_3, s_5$.}\label{fig:regions}
\end{figure}

\begin{prop}\label{prop:basic-facts} Assume  $S=(b+\Z^n)\cap Q$, where $b\notin\Z^n$ and $Q$ is a rational polyhedron, and let  $K$ be a maximal $S$-free polyhedron with $0\in\intr(K)$. The following hold:
\begin{enumerate}[\upshape(i)]
\item $\lin(T(s)) = \rec(T(s)) =L_K$ for every $s \in K\cap S$.
\item $T(s)=T(s')$ for every $s,s'\in \R^n$ such that $s-s'\in L_K$.
\item $T(S,K)$ is a union of finitely many polyhedra.
\end{enumerate}
\end{prop}

\begin{proof}
\begin{enumerate}[(i)]
\item From the description of $T(s)$, $$\rec(T(s)) = \{r\in\R^n\colon (a_i - a_{k(s)}) r \leq 0, \;\; (a_{k(s)} - a_i) r \leq 0 \; \mbox{ for all } i \in I\}$$ where ${k(s)}$ is the index of the facet of $K$ containing $s$. Hence, for every $r\in \rec(T(s))$, we obtain $a_i r = a_{k(s)} r$ for all $i\in I$. This shows that $\rec(T(s)) = \lin(T(s)) = L_K$.
\item Observe that
$$\begin{array}{rcl}  r \in T(s') & \Leftrightarrow & (a_i - a_{k(s')}) r \leq 0, \;\; (a_i - a_{k(s')}) (s' -r) \leq 0 \;\;\quad \forall i \in I \\
& \Leftrightarrow & (a_i - a_{k(s')}) r \leq 0, \;\; (a_i - a_{k(s')}) (s' + (s - s') - r) \leq 0 \;\;\quad \forall i \in I \\
& \Leftrightarrow & (a_i - a_{k(s)}) r \leq 0, \;\; (a_i - a_{k(s)}) (s - r) \leq 0 \;\;\quad \forall i \in I \\
& \Leftrightarrow & r \in T(s)
\end{array}$$
where the second equivalence follows from the fact that $s - s' \in L_K$ and so $a_i (s-s') = a_k (s-s')$, and in the third equivalence we use the fact that $k(s)$ and $k(s')$ can be chosen to be the same as $s-s' \in L_K$.
\item
 Since $K$ is full-dimensional,  $K$ satisfies either condition (i) or condition (ii) of Theorem \ref{thm:S-free}. In both cases, if we define $L:=\langle\rec(K\cap \conv(S))\rangle$, then $L$ is a lattice subspace contained in $\lin(K)$. Since $\lin(K)=\{r\in \R^n\colon a_ir=0 \mbox{ for all } i\in I \}$, we have that $L\subseteq \lin(K)\subseteq L_K$.


Since $L=\langle\rec(K\cap\conv(S))\rangle$, it follows that $\proj_{L^\perp}(K)\cap \proj_{L^\perp}(\conv(S))$ is a polytope.

Given two elements $s,s'\in S$ whose orthogonal projections onto $L^\perp$ coincide, it follows that $s-s'\in L\subseteq L_K$, and therefore by (ii) $T(s)=T(s')$. Then the number of sets $T(s)$, $s\in S\cap K$, is at most the cardinality of $\proj_{L^\perp}(S\cap K)$.

Let $S'$  and $b'$ be the orthogonal projections of $S$ and $b$ onto $L^\perp$. Since $L$ is a lattice subspace, $S' = (b' + \Lambda)\cap \proj_{L^\perp}(\conv(S))$, where $\Lambda$ is a lattice in $L^\perp$.   Since $\proj_{L^\perp}(K)\cap \proj_{L^\perp}(\conv(S))$ is a polytope, $\proj_{L^\perp}(K)\cap \proj_{L^\perp}(\conv(S))\cap (b'+\Lambda)$ is finite. Furthermore, since $\proj_{L^\perp}(S\cap K)\subseteq \proj_{L^\perp}(K)\cap \proj_{L^\perp}(\conv(S))\cap (b'+\Lambda)$, it follows that $\proj_{L^\perp}(S\cap K)$ is a finite set.

We conclude that the family of polyhedra $T(s)$, $s\in S\cap K$, has a finite number of elements, thus $T(S,K)=\bigcup_{s\in S \cap K}T(s)$ is the union of a finite number of polyhedra.
\end{enumerate}
\end{proof}

\subsubsection{The covering property}

We assume here $S=(b+\Z^n)\cap Q$, where $Q$ is a rational polyhedron. As mentioned earlier, the minimal valid functions for $S$ are in one-to-one correspondence with maximal $S$-free convex sets containing the origin in their interior. For any such maximal $S$-free convex set $K$, we refer to the lifting region $T_\psi$ for the minimal cut-generating function $\psi$ corresponding to $K$ by $T(S,K)$, to emphasize the dependence on $S$ and $K$. We say $T(S,K)$ has the {\em covering property} if $T(S,K) + W_S = \R^n$. When $S$ is clear from the context, we will also say that $K$ has the {\em covering property} if $T(S,K)$ has the covering property. 

\paragraph{Results of this section and their importance for discrete optimization.} The main results presented in this section are three operations that preserve the covering property, namely, translations, the so-called {\em coproduct} and {\em limit} operations. Moreover, a necessary and sufficient condition is presented in Theorem~\ref{thm:unique-integer} to characterize which pyramids in a particular family have the covering property.

The importance of these results in terms of cutting planes is the following. The pyramids in Theorem~\ref{thm:unique-integer} and classification of the covering property in $\R^2$ (see~\cite{basu-paat-lifting}) provide a ``base set" of maximal $S$-free convex sets with the covering property. By iteratively applying the three operations of translations, coproducts and limits, we can then build a vast (infinite) list of maximal $S$-free convex sets (in arbitrarily high dimensions) with the covering property, enlarging this ``base set". Not only does this recover all the previously known sets with the covering property, it vastly expands this list. Earlier, ad hoc families of $S$-free convex sets were proven to have the covering property ---now we have generic operations to construct infinitely many families. These make a contribution in the modern thrust on obtaining efficiently computable formulas for computing cutting planes, by giving a much wider class of maximal $S$-free sets whose lifting regions have the covering property.

\paragraph{The covering property is preserved under translations.} Let $K$ be a maximal $S$-free polyhedron with the origin in its interior. In~\cite{basu-paat-lifting}, Basu and Paat prove the following theorem.

\begin{theorem}\label{thm:invariant} Let $t \in \R^n$ be such that $K+t$ also contains the origin in its interior. Then, $T(S,K) + W_S = \R^n$ if and only if $T(S+t, K+t) + W_{S+t} = \R^n$.  \end{theorem}
In other words, the covering property is preserved under translations.

This theorem is not obvious, since both the function $\psi$  defined according to~\eqref{eq:psi-formula} and the set $T(S,K)$ change in a non-trivial way when we translate $S$ and $K$.  For the case $S=b+\Z^n$, this result was first proved when $K$ is a maximal $S$-free {\em simplicial} polytope~\cite{basu2012unique} and then for any maximal $S$-free polyhedron in~\cite{averkov-basu-lifting}.
Both proofs are based on volume arguments, which do not seem easily extendable to the more general case $S=(b+\Z^n)\cap Q$ for a rational polyhedron $Q$.
The generalization to this setting was obtained in~\cite{basu-paat-lifting} by using the {\em invariance of domain} as the main tool. This is an important result in algebraic topology, first proved by Brouwer~\cite{brouwer1911beweis, Dold1995}.

\begin{theorem}[Invariance of Domain]\label{thm:invariance-domain}
If $U$ is an open subset of $\R^n$ and $f \colon U \to \R^n$ is an injective, continuous map, then $f(U)$ is open and $f$ is a homeomorphism between $U$ and $f(U)$.
\end{theorem}

An outline of the proof of the invariance of the lifting region under translations when $S=(b+\Z^n)\cap Q$ (Theorem~\ref{thm:invariant}) is now provided.
We define $S'=S+t$ and $K'=K+t$.

\begin{enumerate}
\item Let $\{x \in \R^n\colon a_ix\le1,\,i\in I\}$ be an irredundant description of a maximal $S$-free polyhedron $K$ with $0\in\intr(K)$.
 For each $k\in I$, define the affine function $f_k$ that maps the affine hyperplane $H= \{r \in \R^n\colon a_k r = 1\}$ to $H+t$.
\item For every $s \in K\cap S$ and $w \in W_S$, define the polyhedron $K_{s,w} = T(s) + w$ and define the map $f_{s,w}\colon K_{s,w} \to \R^n$ as $f_{s,w}(x) = f_{k(s)}(x-w) + w$, where ${k(s)} \in I$ is such that $a_{k(s)} s = 1$. Since $T(S,K) + \Z^n = \R^n$, we have

\begin{equation*}
\bigcup_{s \in K\cap S, w \in W_S} K_{s,w} = T(S,K) + \Z^n = \R^n.
\end{equation*}

\item It is shown that for any two pairs $s_1, w_1$ and $s_2, w_2$ we have that $f_{s_1, w_1}(x) = f_{s_2, w_2}(x)$ for all $x \in K_{s_1,w_1}\cap K_{s_2,w_2}$. Thus, the different $f_{s,w}$'s can be ``stitched together" to give a well-defined map $f \colon\R^n \to \R^n$ such that $f$ restricted to $K_{s,w}$ is equal to $f_{s,w}$. Since each $f_{s,w}$ is an invertible affine map on $K_{s,w}$, and any bounded set intersects only finitely many polyhedra $K_{s,w}$, it can be shown that $f$ is an injective continuous map on $\R^n$.

\item It is also shown that $f_{s,w}(T(s) + w) = T(s+t) + w$ for every $s \in S\cap K$ and $w \in W_S$. In other words, the affine function $f_{s,w}$ maps a spindle in $T(S,K)$ (translated by the vector $w$) into the corresponding spindle in $T(S',K')$ (translated by the same vector $w$).

\item By Theorem~\ref{thm:invariance-domain}, $f(\R^n)$ is open. On the other hand, $f(\R^n) = T(S', K') + W_S$ can be shown to be closed because $T(S', K')$ is the union of finite many polyhedra translated by a lattice $W_S = \lin(\conv(S)) \cap \Z^n$. Since $\R^n$ is connected, the only non-empty closed and open subset of $\R^n$ is $\R^n$ itself. Thus, $f(\R^n) =\R^n$.

\item Finally one observes that $W_S = \lin(\conv(S)) \cap \Z^n = W_{S'}$.  Therefore, $$\begin{array}{rcl} T(S',K') + W_{S'} & = & T(S',K') + W_S \\
& = & \bigcup_{s' \in K'\cap S', w \in W_S} T(s') + w \\
& = & \bigcup_{s \in K\cap S, w \in W_S} T(s + t) + w \\
& = & \bigcup_{s \in K\cap S, w \in W_S} f_{s,w}(T(s) + w) \\
& = & f\big( \bigcup_{s \in K\cap S, w \in W_S} (T(s) + w)\big) \\
& = & f(T(S,K) + W_S) \\
& = & f(\R^n) \\
& = & \R^n
\end{array}
$$
where the fourth equality follows from Step 3.
\end{enumerate}

\paragraph{The coproduct and limit operations preserve the covering property.}

We define an operation on polytopes that preserves the covering property. Namely, given two polytopes $K_1\in \R^{n_1}$ and $K_2\in \R^{n_2}$ with $K_i$ containing the origin in its interior for $i=1,2$, we define the {\em coproduct} $K_1 \diamond K_2$ as follows. Let $K_i^\bullet$ be the (inclusion-wise) smallest prepolar for $K_i$, $i=1,2$. Define $K_1 \diamond K_2 = (K_1^\bullet \times K_2^\bullet)^\circ$ where we remind the reader that $V^\circ$ denotes the polar of a set $V$, and $X \times Y$ denotes the cartesian product of $X$ and $Y$. Let $n = n_1 + n_2$ and for $i\in \{1,2\}$, let $S_i = (b_i+\Z^{n_i})\cap Q_i$, where $Q_i \subseteq \R^{n_i}$ is a rational polyhedron and $b_i \in \R^{n_i}\setminus\Z^{n_i}$. Then $S_1\times S_2 = ((b_1,b_2)+(\Z^{n_1}\times \Z^{n_2}))\cap (Q_1\times Q_2)$. Therefore, it is reasonable to speak of $S_1\times S_2$-free convex sets. If $K_i$ is maximal $S_i$-free such that $T(S_i, K_i)$ has the covering property for $i\in \{1,2\}$, then for any $0\le \mu \le 1$, $\frac{K_1}{\mu}\diamond \frac{K_2}{1-\mu}$ is maximal $S_1\times S_2$-free and $T(S_1 \times S_2,\frac{K_1}{\mu}\diamond \frac{K_2}{1-\mu})$ has the covering property. This is proved in~\cite{basu-paat-lifting}; it was first shown in~\cite{averkov-basu-lifting} for the case when $S_i=b_i+\Z^{n_i}$ for $i\in\{1,2\}$. This is a very useful operation to create higher dimensional maximal $S$-free convex sets with the covering property by ``gluing'' together lower dimensional such sets.

When $S=(b+\Z^n)\cap Q$ for some rational polyhedron $Q$, it is shown in~\cite{basu-paat-lifting} that if a sequence of maximal $S$-free convex sets, all of whose lifting regions have the covering property, converges (in some precise sense) to a maximal $S$-free convex set, then the limit set also has the covering property. This is a generalization of a result from~\cite{averkov-basu-lifting}.

\paragraph{Special polytopes that have the covering property.} In this last part we assume $S=b+\Z^n$. We define a pyramid as the convex hull of an $(n-1)$-dimensional polytope $B$ and a point $v\notin\aff(B)$. $v$ is called the {\em apex} and $B$ is the {\em base} of the pyramid.

\begin{theorem}\label{thm:unique-integer}\cite{averkov-basu-lifting} Assume  $S=b+\Z^n$ for some $b\not\in\Z^n$, and let $K$ be a maximal $S$-free pyramid in $\R^n$ ($n \geq 2$) such that every facet of $K$ contains exactly one point from $S$ in its relative interior. $K$ has the covering property if and only if $K$ is the image of $\conv\{0, ne^1, \ldots, ne^n\}$ under an affine unimodular transformation.
\end{theorem}

Assume first that $K$ is $\conv\{0, ne^1, \ldots, ne^n\}$ (after applying an affine unimodular transformation). For each $i=1, \ldots, n$, consider the translation $K-ne^i$ and the spindle formed by the facet containing the point $e^i - ne^i$; this spindle is the unimodular transformation of the cube $[0,1]^n$. Moreover, the spindle in $K$ with respect to the point $(1, 1, \ldots, 1)$ is the cube $[0,1]^n$. Thus, each of these translation vectors leads to a lifting region which covers $\R^n$ by integer translates. An extension of the translation invariance property can be used to show that this implies $K$ has the covering property.\smallskip

We now outline the proof of the other implication.
\begin{enumerate}
\item Consider a translation of $K$ such that the apex of $K$ becomes the origin, and let $T$ be the spindle corresponding to the single integer point on the base of the pyramid. Using an extension of the translation invariance property proved above, one can show that $T$ covers $\R^n$ by integer translates. It is also not hard to show that two integer translates of $T$ cannot intersect in the interior. Thus $T$ actually tiles $\R^n$ by integer translates.

\item For a polytope $K$ and any face $F$ of $K$ of dimension $n-2$, the {\em belt} corresponding to $F$ in $K$ is the set of all facets that contain a translate of $F$ or $-F$. The {\em Venkov-Alexandrov-McMullen theorem} from the geometry of numbers states:
\begin{theorem} [{\cite[Theorem~32.2]{gruber}}]\label{thm:venkov}
Let $K$ be a compact convex set with nonempty interior that translatively tiles $\R^n$. Then the following assertions hold:
\begin{enumerate}[\upshape(a)]
	\item $K$ is a centrally symmetric polytope.
	\item All facets of $K$ are centrally symmetric.
	\item Every belt of $K$ is either of length $4$ or $6$.
\end{enumerate}
\end{theorem}
This implies that $T$ is centrally symmetric with centrally symmetric facets.

\item It can be shown that since $T$, in this special case, is a spindle with centrally symmetric facets, every belt of $T$ is of length 4. This, in turn, can be used to show that every face of dimension $n-2$ is centrally symmetric. {\em McMullen's characterization of zonotopes}~\cite{mcmullen} then implies that $T$ is a zonotope. (A zonotope is the Minkowski sum of finitely many line segments; equivalently, a zonotope is the image under an affine map ---not necessarily invertible--- of a cube.)

\item Combinatorial geometry of zonotopes can be used to show that any zonotope whose belts have all length 4 is a parallelotope (i.e., the invertible affine image of a cube). Thus $T$ is a parallelotope and $K$ is a simplex.

\item Since $T$ tiles $\R^n$ by integer translates, $T$ has volume $1$. Moreover, it can be shown that there exists a translation vector $t\in \R^n$ such that the polytope $2T+t$ is centrally symmetric about the origin and the only integer point in its interior is the origin. Moreover, each facet of $2T+t$ contains exactly one integer point in its relative interior. We now appeal to the {\em Minkowski-Haj\'os theorem}:

\begin{theorem}[\cite{GruberLekkerkerker-Book87}, Section 12.4, Chapter 2]\label{thm:minkowski3}
Let $S$ be a $0$-symmetric parallelotope such that each it has no integer point in its interior besides $0$, and suppose that $S$ has volume $2^n$. Then there exists a unimodular transformation $U$ such that after applying $U$, $S$ will have two parallel facets given by $-1 \leq x_1 \leq 1$.
\end{theorem}

We use this theorem to prove the following lemma.

\begin{lemma}\label{lem:extremal_parallel}
Let $S$ be a $0$-symmetric parallelotope with no integer point in its interior besides $0$, and suppose that $S$ has volume $2^n$. If every facet of $S$ has exactly one integer point in its relative interior, then $S$ is a unimodular transformation of the cube $[-1,1]^n$.
\end{lemma}

\begin{proof}
We prove this by induction on the dimension $n$. For $n=1$, this is trivial. Consider $n \geq 2$. The Minkowski--Haj\'os theorem (Theorem~\ref{thm:minkowski3}) implies that we can apply a unimodular transformation such that $S$ = $\conv\{(S\cap \{x_1 = -1\}) \cup (S\cap \{x_1 = 1\})\}$. Note that $S\cap \{x_1 = -1\}$, $S\cap \{x_1 = 0\}$ and $S\cap \{x_1 = 1\}$ are all translations of each other. Therefore, $2^n = \vol(S)= 2\vol(S\cap \{x_1 = 0\})$ (here we measure volume of $S\cap \{x_1 = 0\}$ in the $(n-1)$-dimensional linear space $x_1 = 0$). So $S\cap \{x_1 = 0\}$ has volume $2^{n-1}$. Therefore, $S\cap \{x_1 = 0\}$ is also a $0$-symmetric parallelotope in the linear space $x_1=0$ with volume $2^{n-1}$, and its only integer point is the origin. If any facet $F$ of  $S\cap \{x_1 = 0\}$ contains two or more integer points in its relative interior, then the facet of $S$ passing through $F$ will contain these integer points in its relative interior, in contradiction to the hypothesis of the theorem. Therefore, every facet of $S\cap \{x_1 = 0\}$ contains at most one integer point in its relative interior. By the induction hypothesis, $S\cap \{x_1 = 0\}$ is equivalent to the cube $\{-1 \leq x_i \leq 1,\; i= 2, \ldots, n\} \cap \{x_1 = 0\}$. Recall that $S\cap \{x_1 = -1\}$ and $S\cap \{x_1 = 1\}$ are translations of $S\cap \{x_1 = 0\}$. Since $S\cap \{x_1 = 0\}$ is equivalent to the cube $\{-1 \leq x_i \leq 1,\; i= 2, \ldots, n\} \cap \{x_1 = 0\}$, any translation by a non-integer vector $(x_1, x_2, \ldots, x_n)$ with $x_1\in \Z$ will contain at least two integer points in its relative interior. But the facets $S\cap \{x_1 = -1\}$ and $S\cap \{x_1 = 1\}$ contain at most one integer point in their relative interior. Therefore, they are in fact integer translates of $S\cap \{x_1 = 0\}$. This proves the lemma.
\end{proof}

We know that $2T+t$ satisfies the hypothesis of Lemma~\ref{lem:extremal_parallel} and therefore $2T+t$ is a unimodular transformation of the cube $[-1,1]^n$. This can be used to show that $K$ is the image of $\conv\{0, ne^1, \ldots, ne^n\}$ under an affine unimodular transformation.

\end{enumerate}

For further details, we refer the reader to~\cite[Section 6]{averkov-basu-lifting} and~\cite[Section 3]{basu2012unique}.


\bibliographystyle{plain}
\bibliography{full-bib}
\end{document}